\documentclass[12pt]{amsart}

\usepackage{amscd}
\usepackage{amsthm}
\usepackage{amsmath}
\usepackage{amsfonts}
\usepackage{amssymb}
\usepackage[all]{xy}
\usepackage{graphicx}
\usepackage{verbatim}
\usepackage{color}

\usepackage{comment}

\usepackage{MnSymbol}

\newtheorem{theorem}{Theorem}[section]
\newtheorem{lemma}[theorem]{Lemma}
\newtheorem{proposition}[theorem]{Proposition}
\newtheorem{corollary}[theorem]{Corollary}

\theoremstyle{definition}
\newtheorem{definition}[theorem]{Definition}

\newtheorem{remark}[theorem]{Remark}
\newtheorem{notation}[theorem]{Notation}
\newtheorem{claim}[theorem]{Claim}

\newcommand{\func}[1]{\operatorname{#1}}

\def\Id{{\rm Id}}

\def\int{{\sf Int}}
\def\cl{{\sf Cl}}

\newcounter{num}

\newcommand{\thd}{{\twoheaddownarrow}}
\newcommand{\thu}{{\twoheaduparrow}}
\newcommand{\down}{{\downarrow}}
\newcommand{\up}{{\uparrow}}

\begin{document}

\title{De Vries powers: A generalization of Boolean powers for compact Hausdorff spaces}
\author{G.~Bezhanishvili, V.~Marra, P.~J.~Morandi, B.~Olberding}

\date{}

\begin{abstract}
We generalize the Boolean power construction to the setting of compact Hausdorff spaces. This is done by replacing Boolean algebras with de Vries algebras (complete Boolean algebras enriched with proximity) and Stone duality with de Vries duality. For a compact Hausdorff space $X$ and a totally ordered algebra $A$, we introduce the concept of a finitely valued normal function $f:X\to A$. We show that the operations of $A$ lift to the set $FN(X,A)$ of all finitely valued normal functions, and that there is a canonical proximity relation $\prec$ on $FN(X,A)$. This gives rise to the de Vries power construction, which when restricted to Stone spaces, yields the Boolean power construction.

We prove that de Vries powers of a totally ordered integral domain $A$ are axiomatized as proximity Baer Specker $A$-algebras, those pairs $(S,\prec)$, where $S$ is a torsion-free $A$-algebra generated by its idempotents that is a Baer ring, and $\prec$ is a proximity relation on $S$. We introduce the category of proximity Baer Specker $A$-algebras and proximity morphisms between them, and prove that this category is dually equivalent to the category of compact Hausdorff spaces and continuous maps. This provides an analogue of de Vries duality for proximity Baer Specker $A$-algebras.
\end{abstract}


\maketitle

\pagestyle{myheadings}
\markboth{G.~Bezhanishvili, V.~Marra, P.~J.~Morandi, B.~Olberding}{De Vries powers: A generalization of Boolean powers for compact Hausdorff spaces}

\section{Introduction}

For an algebra $A$ of a given type and a Boolean algebra $B$, the \emph{Boolean power of $A$ by $B$} is the algebra $C(X,A_\mathrm{disc})$ of all continuous functions from the Stone space $X$ of $B$ to $A$, where $A$ is given the discrete topology and the operations of $A$ are lifted to $C(X,A_\mathrm{disc})$ pointwise (see, e.g., \cite{BN80,BS81}). For convenience, we also refer to $C(X,A_\mathrm{disc})$ as the \emph{Boolean power of $A$ by $X$}. Boolean powers turned out to be a very useful tool in universal algebra, where they have been used to transfer results about Boolean algebras to other varieties \cite{BS81}.

There is no obvious way to generalize the Boolean power construction to compact Hausdorff spaces. Since $X$ is compact and $A$ is discrete, each $f\in C(X,A_\mathrm{disc})$ is finitely valued, and gives a partition of $X$ into finitely many clopen (closed and open) sets. So if there are not  enough clopens in $X$, then $C(X,A_\mathrm{disc})$ is not representative enough. For example, if $X=[0,1]$, then $C(X,A_\mathrm{disc})$ degenerates to simply $A$. The goal of this article is to generalize the Boolean power construction in such a way that it encompasses compact Hausdorff spaces. For this, instead of working with clopen sets, which form a basis only in the zero-dimensional case, we will work with regular open sets, which form a basis for any compact Hausdorff space.

One of the most natural generalizations of Stone duality to compact Hausdorff spaces is de Vries duality \cite{deV62}. We recall that a binary relation $\prec$ on a Boolean algebra $B$ is a \emph{proximity} if it satisfies the following axioms:
\begin{enumerate}
\item[(DV1)] $1\prec 1$.
\item[(DV2)] $a\prec b$ implies $a\le b$.
\item[(DV3)] $a\le b\prec c\le d$ implies $a\prec d$.
\item[(DV4)] $a\prec b,c$ implies $a\prec b\wedge c$.
\item[(DV5)] $a\prec b$ implies $\neg b\prec \neg a$.
\item[(DV6)] $a\prec b$ implies there is $c\in B$ such that $a\prec c\prec b$.
\item[(DV7)] $a\ne 0$ implies there is $0\ne b\in B$ such that $b\prec a$.
\end{enumerate}
A \emph{proximity Boolean algebra} is a pair $(B,\prec)$, where $B$ is a Boolean algebra and $\prec$ is a proximity on $B$, and a \emph{de Vries algebra} is a proximity Boolean algebra such that $B$ is complete as a Boolean algebra.

By de Vries duality, each compact Hausdorff space $X$ gives rise to the de Vries algebra $(\mathcal{RO}(X),\prec)$, where $\mathcal{RO}(X)$ is the complete Boolean algebra of regular open subsets of $X$, the Boolean operations on $\mathcal{RO}(X)$ are given by $\bigvee U_i = \int\left(\cl\left(\bigcup_i U_i\right)\right)$, $\bigwedge U_i = \int\left(\bigcap_i U_i\right)$, and $\lnot U = \int(X\backslash U)$, and the proximity is given by $U\prec V$ iff ${\sf Cl}(U)\subseteq V$, where $\int$ and $\cl$ are the interior and closure operators. Moreover, each de Vries algebra $(B,\prec)$ is isomorphic to the de Vries algebra $(\mathcal{RO}(X),\prec)$ for a unique (up to homeomorphism) compact Hausdorff space $X$. This 1-1 correspondence extends to a dual equivalence between the categories of de Vries algebras and compact Hausdorff spaces. To define the category of de Vries algebras, we recall that a map $\sigma:B\to C$ between proximity Boolean algebras is a \emph{de Vries morphism} provided
\begin{enumerate}
\item[(M1)] $\sigma(0)=0$.
\item[(M2)] $\sigma(a\wedge b)=\sigma(a)\wedge\sigma(b)$.
\item[(M3)] $a\prec b$ implies $\neg\sigma(\neg a)\prec\sigma(b)$.
\item[(M4)] $\sigma(a)$ is the least upper bound of $\{\sigma(b):b\prec a\}$.
\end{enumerate}
Note that function composition of two de Vries morphisms need not be a de Vries morphism because it need not satisfy (M4). Nevertheless, the de Vries algebras and de Vries morphisms between them form a category {\bf DeV}, where the composition $\rho\star\sigma$ of two de Vries morphisms $\sigma:B_1\to B_2$ and $\rho:B_2\to B_3$ is given by
\[
(\rho \star \sigma)(a)=\bigvee\{\rho\sigma(b): b\prec a\}.
\]
Each continuous function $\varphi:X\to Y$ between compact Hausdorff spaces $X,Y$ gives rise to the de Vries morphism $\widehat{\varphi}:\mathcal{RO}(Y)\to\mathcal{RO}(X)$, where $\widehat{\varphi}(U)={\sf Int}\left({\sf Cl}\left(\varphi^{-1}(U)\right)\right)$ for each $U\in\mathcal{RO}(Y)$. Moreover, each de Vries morphism between de Vries algebras comes about this way. The upshot of all this is that {\bf DeV} is dually equivalent to the category {\bf KHaus} of compact Hausdorff spaces and continuous maps, which is one of the key results of \cite{deV62}.

We wish to use de Vries duality to define the de Vries power of an algebra by a compact Hausdorff space the same way Stone duality is used to define the Boolean power of an algebra by a Stone space. As a motivating example, let $X$ be a compact Hausdorff space and let $f:X\to\mathbb R$ be a finitely valued function. If $f$ is continuous, then $f^{-1}(a,\infty)$ is clopen in $X$ for each $a\in\mathbb R$. On the other hand, we show that $f^{-1}(a,\infty)$ is regular open for all $a \in \mathbb{R}$ iff $f$ is a normal function, where we recall that a lower semicontinuous function $f$ is normal provided $f^{-1}(-\infty,a)$ is a union of regular closed sets for each $a\in\mathbb R$ \cite[Sec.~3]{Dil50}. Since for a finitely valued function $f$, we have $f^{-1}(a,\infty)=f^{-1}[b,\infty)$ for some $b>a$ (and $f^{-1}(-\infty,a)=f^{-1}(-\infty,b]$ for some $b<a$), this observation allows us to generalize the concept of a finitely valued normal function as follows.

Let $A$ be a totally ordered algebra of a given type, let $X$ be a compact Hausdorff space, and let $f:X\to A$ be a finitely valued function. We call $f$ \emph{normal} if $f^{-1}(\up a)$ is regular open in $X$ for each $a\in A$, where $\up a=\{b\in A:a\le b\}$. Let $FN(X,A)$ be the set of finitely valued normal functions from $X$ to $A$. For a finitely valued function $f:X\to A$, we introduce the concept of \emph{normalization} of $f$, and show that normalization lifts the operations of $A$ to $FN(X,A)$. Thus, $FN(X,A)$ has the algebra structure of $A$. In addition, $FN(X,A)$ has a canonical proximity given by $f\prec g$ iff $f^{-1}(\up a)\prec g^{-1}(\up a)$ in $\mathcal{RO}(X)$ for each $a\in A$. We call the pair $(FN(X,A),\prec)$ the \emph{de Vries power of $A$ by $X$}. Equivalently, if $(B,\prec)$ is a de Vries algebra and $X$ is its dual compact Hausdorff space, then we call $(FN(X,A),\prec)$ the \emph{de Vries power of $A$ by $(B,\prec)$}. We show that when $X$ is a Stone space, this construction yields the Boolean power construction.

The main goal of this article is to axiomatize de Vries powers of a totally ordered integral domain, thus including such classic cases as $\mathbb Z,\mathbb Q$, and $\mathbb R$. Our results generalize several known results in the literature. Boolean powers of $\mathbb Z$ were studied by Ribenboim \cite{Rib69}. They turn out to be exactly the Specker $\ell$-groups introduced and studied by Conrad \cite{Con74}. On the other hand, Boolean powers of $\mathbb R$ are the Specker $\mathbb R$-algebras introduced and studied in \cite{BMO13a}. The category of Specker $\mathbb R$-algebras is dually equivalent to the category of Stone spaces, and this duality can be thought of as an economic version of Gelfand-Neumark-Stone duality in the particular case of Stone spaces \cite[Rem.~6.9]{BMO13a}. In \cite{BMMO13a}, these results were generalized to axiomatize Boolean powers of a commutative ring.

Let $A$ be a commutative ring with $1$, let $S$ be a commutative $A$-algebra with $1$, and let $\mathrm{Id}(S)$ be the Boolean algebra of idempotents of $S$. A nonzero $e\in\mathrm{Id}(S)$ is \emph{faithful} provided $ae=0$ implies $a=0$ for each $a\in A$. We call $S$ a \emph{Specker $A$-algebra} if $S$ is generated as an $A$-algebra by a Boolean subalgebra $B$ of $\mathrm{Id}(S)$ whose nonzero elements are faithful. In case $A$ is an integral domain, $S$ is a Specker $A$-algebra iff $S$ is generated as an $A$-algebra by $\mathrm{Id}(S)$ and $S$ is torsion-free as an $A$-module \cite[Prop.~4.1]{BMMO13a}. By \cite[Thm.~2.7]{BMMO13a}, Boolean powers of $A$ are precisely Specker $A$-algebras. Moreover, if $A$ is a domain (or more generally if $A$ is an indecomposable ring; that is, if $\mathrm{Id}(A)=\{0,1\}$), then the category of Specker $A$-algebras is equivalent to the category of Boolean algebras, and is dually equivalent to the category of Stone spaces \cite[Thm.~3.8 and Cor.~3.9]{BMMO13a}.

In this article, for a totally ordered domain $A$, we enrich the concept of a Specker $A$-algebra to that of a proximity Specker $A$-algebra, and show that a de Vries power of a totally ordered domain is precisely a proximity Specker $A$-algebra that is also a Baer ring. We prove that each proximity Specker $A$-algebra $(S,\prec)$ can be represented as a dense subalgebra of $(FN(X,A),\prec)$ for a unique (up to homeomorphism) compact Hausdorff space $X$. We also prove that $(S,\prec)$ is isomorphic to $(FN(X,A),\prec)$ iff $S$ is a Baer ring. We introduce proximity morphisms between proximity Specker $A$-algebras, and show that the proximity Baer Specker $A$-algebras with proximity morphisms between them form a category ${\bf PBSp}_A$ that is equivalent to {\bf DeV} and is dually equivalent to {\bf KHaus}. In fact, the functor ${\bf KHaus}\to{\bf PBSp}_A$ is the de Vries power functor, while the functor ${\bf PBSp}_A\to{\bf KHaus}$ associates with each proximity Baer Specker $A$-algebra $(S,\prec)$, the compact Hausdorff space of ends of $(S,\prec)$. The obtained duality provides an analogue of de Vries duality for proximity Baer Specker $A$-algebras.

The article is organized as follows. In Section 2 we introduce finitely valued normal functions and establish their basic properties. In Section 3, for a totally ordered algebra $A$, we generalize the notion of a Boolean power of $A$ to that of a de Vries power of $A$. In Section 4 we specialize to the case of a totally ordered integral domain $A$, introduce the notion of a proximity Specker $A$-algebra, and show that a de Vries power of $A$ is a proximity Baer Specker $A$-algebra. In Section 5 we prove our main representation theorem that every proximity Specker $A$-algebra $(S,\prec)$ embeds in a de Vries power of $A$, and that the embedding is an isomorphism iff $S$ is Baer. In Section 6 we introduce proximity morphisms. For proxmity Specker $A$-algebras $(S, \prec)$ and $(T, \prec)$, we prove that there is a 1-1 correspondence between proximity morphisms $S \to T$, de Vries morphisms $\Id(S) \to \Id(T)$, and continuous maps $Y \to X$, where $X$ and $Y$ are the de Vries duals of $\Id(S)$ and $\Id(T)$, respectively. In Section 7 we introduce ends of a proximity Specker $A$-algebra $(S, \prec)$, give several characterizations of ends, and show that the space of ends of $(S, \prec)$ is homeomorphic to the de Vries dual of $\Id(S)$. Finally, in Section 8 we prove that the proximity Baer Specker $A$-algebras form a category that is equivalent to the category of de Vries algebras and is dually equivalent to the category of compact Hausdorff spaces.

\section{Finitely valued normal functions}

Throughout this section we assume that $X$ is a compact Hausdorff space and $A$ is a totally ordered set. In Section~3 we specialize to the case in which $A$ is a totally ordered algebra of a given type, and in Section~4 to the case when it is an integral domain. In this section though the algebraic structure of $A$ plays no role.

For $a\in A$, let $\up a=\{b\in A:a\le b\}$, $\down a=\{b\in A:b\le a\}$, and $[a,b]=\up a\cap\down b=\{x\in A:a\le x\le b\}$. We write $a<b$ provided $a\le b$ and $a\ne b$. We topologize $A$ with the interval topology. In this topology closed intervals $[a,b]$ form a basis of closed sets.

\begin{notation} \label{notation}
\begin{enumerate}
\item[]
\item We denote by $F(X)= F(X,A)$ the set of all finitely valued functions from $X$ to $A$; that is, $F(X)$ is the set of all functions $f : X \to A$ whose image is finite.
\item We denote by $FC(X)=FC(X,A)$ the set of all finitely valued continuous functions from $X$ to $A$, where $A$ has the interval topology. As follows from \cite[Prop.~5.4]{BMMO13a}, $FC(X,A)=C(X,A_{\mathrm{disc}})$.
\item For nonempty $X$, each $a \in A$ gives rise to the constant function on $X$ whose value is $a$. Clearly this function is in $FC(X)$, and we will view $A$ as a subset of $FC(X)$.
\end{enumerate}
\end{notation}

Under the pointwise order, $F(X)$ is a lattice, where the join and meet operations are also pointwise: $\sup(f,g)(x) = \max\{f(x),g(x)\}$ and $\inf(f,g)(x) = \min\{f(x),g(x)\}$. Clearly $FC(X)$ is a sublattice of $F(X)$.

We make frequent use of the simple observation that a finitely valued function on $X$ can alternatively be viewed as a function from $A$ to the powerset of $X$. We formalize this in the following lemma. If $U$ is a subset of $X$, we denote by $\chi_U$ the characteristic function of $U$.

\begin{lemma}\label{easy first}
\begin{enumerate}
\item[]
\item If $f \in F(X)$ and $a_0 < \cdots < a_n$ are the values of $f$, set $U_i = f^{-1}(\up a_i)$ for $0 \le i \le n$ and $U_{n+1} = \varnothing$. Then $X = U_0 \supset U_1 \supset  \cdots \supset U_n \supset U_{n+1} = \varnothing$. Moreover, $f(x) = a_i$ iff $x \in U_i - U_{i+1}$, and $f = a_0 + \sum_{i=1}^{n} (a_{i} - a_{i-1}) \chi_{U_{i}}$.
\item Conversely, if $X = U_0 \supset U_1 \supset \cdots \supset U_n \supset U_{n+1} = \varnothing$ and  $a_0 < \cdots < a_n$ are elements of $A$, then the function $f : X \to A$ defined by $f(x) = a_i$ if $x \in U_i - U_{i+1}$ is finitely valued and $f^{-1}(\up a_i) = U_i$.
\end{enumerate}
\end{lemma}

\begin{proof}
Straightforward.
\end{proof}

Therefore, to define a finitely valued function on $X$, it suffices to produce a finite sequence $a_0 < \cdots < a_n$ in $A$ and a finite sequence $X = U_0 \supset U_1 \supset \cdots \supset U_n \supset \varnothing$ of subsets of $X$. The next lemma shows that two elements $f,g \in F(X)$ can be described in a compatible way.

\begin{lemma}\label{refine}
Let $f,g \in F(X)$. If the values of $f$ and $g$ are among $a_0 < \cdots < a_n$ and $a_{n+1} \in A$ satisfies $a_n < a_{n+1}$, then $f(x) = a_i$ if $x \in f^{-1}(\up a_i) - f^{-1}(\up a_{i+1})$ and $g(x) = a_i$ if $x \in g^{-1}(\up a_i) - g^{-1}(\up a_{i+1})$. Furthermore, $f \le g$ iff $f^{-1}(\up a_i) \subseteq g^{-1}(\up a_i)$ for each $i$. Consequently, $f = g$ iff $f^{-1}(\up a_i) = g^{-1}(\up a_i)$ for each $i$.
 \end{lemma}

\begin{proof}
Straightforward.
\end{proof}

In \cite{Dil50} Dilworth described the Dedekind-MacNeille completion of the lattice $C(X,{\mathbb{R}})$ of continuous real-valued functions by means of \emph{normal functions}; that is, lower semicontinuous functions $f : X \to \mathbb{R}$ for which $f^{-1}(-\infty,a)$ is a union of regular closed sets for each $a \in \mathbb{R}$ (see \cite[Thm.\ 3.2]{Dil50}; note that Dilworth worked with upper semicontinuous functions). We adapt Dilworth's notion of normal function to the setting of functions with finitely many values in $A$. To motivate our definition, we first describe finitely valued normal functions in the special case in which $A = {\mathbb{R}}$; this description is not needed later in the paper, but see \cite{BMO13c} for a development of proximity in the setting of real-valued normal functions.

\begin{proposition}\label{prop:2.4}
Let $f : X \to \mathbb{R}$ be finitely valued. The following conditions are equivalent.
\begin{enumerate}
\item $f$ is normal.
\item $f^{-1}(a,\infty)$ is regular open in $X$ for each $a \in \mathbb{R}$.
\item $f^{-1}[a,\infty)$ is regular open in $X$ for each $a \in \mathbb{R}$.
\end{enumerate}
\end{proposition}

\begin{proof}
(2)$\Leftrightarrow$(3): Since $f$ is finitely valued, for each $a\in\mathbb R$ there is $b>a$ with $f^{-1}(a,\infty) = f^{-1}[b,\infty)$; we may choose $b$ to be the smallest value of $f$ greater than $a$ if such a value exists, or else $b$ may be chosen to be any real number larger than $a$. (Similarly, $f^{-1}(-\infty,a)=f^{-1}(-\infty,b]$ for some $b<a$.) From this it is evident that conditions (2) and (3) are equivalent.

(1)$\Rightarrow$(2): It is known that a bounded real-valued function $f$ is lower semicontinuous iff $f^{-1}(a,\infty)$ is open for each $a \in \mathbb{R}$, and that a lower semicontinuous function $f$ is normal iff $f^{-1}(-\infty,a)$ is a union of regular closed sets for each $a \in \mathbb{R}$ \cite[Thm.~3.2]{Dil50}. Suppose $f$ is normal. Let $a \in \mathbb{R}$. Since $f$ is lower semicontinuous, $f^{-1}(a,\infty)$ is open in $X$. Thus, $f^{-1}(-\infty,a] = X - f^{-1}(a,\infty)$ is closed. Because $f$ is finitely valued, there is $c \in \mathbb{R}$ with $f^{-1}(-\infty, a] = f^{-1}(-\infty, c)$. Since $f$ is normal and regular closed sets are closures of open sets, there is a family $\{U_i\}$ of open sets such that $f^{-1}(-\infty, c) = \bigcup_i{\sf Cl}(U_i)$. Let $U=\bigcup_i U_i$, an open set. Clearly $U\subseteq f^{-1}(-\infty, c)$ and, as $f^{-1}(-\infty, c)$ is closed, ${\sf Cl}(U)\subseteq f^{-1}(-\infty, c)$. On the other hand, $U_i\subseteq U$ implies ${\sf Cl}(U_i)\subseteq{\sf Cl}(U)$, so $f^{-1}(-\infty, c)=\bigcup_i{\sf Cl}(U_i)\subseteq{\sf Cl}(U)$. Therefore, $f^{-1}(-\infty, c)={\sf Cl}(U)$, and as $U$ is open, $f^{-1}(-\infty, c) = f^{-1}(-\infty, a]$ is regular closed. Thus, its complement $f^{-1}(a,\infty)$ is regular open.


(2)$\Rightarrow$(1): Suppose that $f^{-1}(a,\infty)$ is regular open in $X$ for each $a\in\mathbb R$. Then it is clear that $f$ is lower semicontinuous. In addition, since $f^{-1}(-\infty,a)=f^{-1}(-\infty,b]$ for some $b<a$, and $f^{-1}(-\infty,b]=X-f^{-1}(b,\infty)$, which is regular closed as $f^{-1}(b,\infty)$ is regular open, we see that $f$ is normal.
\end{proof}

We use this characterization of finitely valued normal functions $f:X\to\mathbb R$ to define finitely valued normal functions $f:X\to A$, where $A$ is an arbitrary totally ordered set.

\begin{definition}
We define a finitely valued function $f : X \to A$ to be {\it normal} provided $f^{-1}(\up a)$ is regular open for each $a \in A$. We denote by $FN(X)=FN(X,A)$ the set of all finitely valued normal functions from $X$ to $A$.
\end{definition}

\begin{remark} \label{decreasing}
If $f\in F(X)$, with $a_0 < \cdots < a_n$ the values of $f$ and $U_i=f^{-1}(\up a_i)$, then $f\in FN(X)$ iff each $U_i$ is regular open. Thus, if each $U_i$ is regular open, then Lemma~\ref{easy first}(1) implies that $a_0 + \sum_{i=1}^n (a_i - a_{i-1}) \chi_{U_i}$ is normal. We will use this fact throughout.
\end{remark}

While a function in $FN(X)$ need not be continuous, the next proposition shows it is continuous on an open dense subset of $X$. This relationship between finitely valued normal functions and continuous functions on open dense subsets is also considered in Proposition~\ref{normalization exists}, and is made more explicit in Theorem~\ref{direct limit theorem}. We remind the reader that we are using the interval topology on $A$, and that $FC(X) = C(X, A_{\func{disc}})$ as pointed out in Notation~\ref{notation}.

\begin{proposition} \label{continuousondense}
If $f \in FN(X)$, then $f$ is continuous on an open dense subset of $X$.
\end{proposition}

\begin{proof}
Let $a_0 < \cdots < a_n$ be the values of $f$, and let $U_i = f^{-1}(\up a_i)$. Then each $U_i$ is regular open. We show that $f$ is continuous on $U := \left(\bigcup_{i=0}^{n-1} \left(U_i - {\sf Cl}(U_{i+1})\right)\right) \cup U_n$, and that this union is open dense in $X$. For continuity, since $f\left(U_i - {\sf Cl}(U_{i+1})\right) = \{a_i\}$ and $f(U_n) = \{a_n\}$, we see that $f$ is constant, hence continuous on the open set $U_i - {\sf Cl}(U_{i+1})$ for each $i$, as well as on the open set $U_n$. Therefore, $f$ is continuous on the open set $U$. To prove density, let $V$ be a nonempty open subset of $X$. There is a smallest $m>1$ with $V \cap U_m = \varnothing$. Then $V \cap U_{m-1} \ne \varnothing$ and $V \cap {\sf Cl}(U_m) = \varnothing$. Therefore, $V \cap U \ne \varnothing$. Thus, $U$ is open dense in $X$.
\end{proof}

\begin{definition}\label{def:2.7}
Let $f\in F(X)$ and let $a_0 < \cdots < a_n$ be the values of $f$. For each $i =0,\ldots, n$, set $U_i = {\sf Int}\left({\sf Cl}\left(f^{-1}(\up a_i)\right)\right)$, and let $U_{n+1} = \varnothing$. Define  $f^\#:X\to A$ by $f^\#(x) = a_i$ provided $x \in U_i - U_{i+1}$. By Lemma~\ref{easy first}(2), $f^\# \in FN(X)$, and we call $f^\#$ the {\it normalization} of $f$.
\end{definition}

\begin{remark} \label{immediate facts}
For $f \in F(X)$, the following facts are immediate:
\begin{enumerate}
\item $f \in FN(X)$ iff $f^\# = f$.
\item If $f=\chi_U$ for $U\subseteq X$, then $f^\#=\chi_{{\sf Int}\left({\sf Cl}(U)\right)}$. More generally, write $f = a_0 + \sum_{i=1}^n (a_i - a_{i-1}) \chi_{U_i}$ as in Lemma~\ref{easy first}(1). Then $f^\# = a_0 + \sum_{i=1}^n (a_i - a_{i-1}) \chi_{\int(\cl(U_i))}$.
\item The image of $f^\#$ is contained in the image of $f$.
\item If $f \in FC(X)$, then $f$ is normal.
\end{enumerate}
\end{remark}

Let $U$ be a nonempty subset of $X$ and let $f \in F(U)$.  Replacing $X$ by $U$ and using the same idea as in Definition~\ref{def:2.7} allows us to define $f^\#\in FN(X)$. Then $f^\#$ is characterized by  $(f^\#)^{-1}(\up a) = {\sf Int}\left({\sf Cl}\left(f^{-1}(\up a)\right)\right)$ for each $a\in A$.

\begin{proposition} \label{normalization exists} \label{restriction} \label{uniqueext}
Let $U$ be an open dense subset of $X$ and let $f \in FC(U)$. Then $f^\#$ is the unique function in $FN(X)$ that restricts to $f$ on $U$.
\end{proposition}

\begin{proof}
Let $a_0 < \cdots < a_n$ be the values of $f$. If $1 \le i \le n$, let $V_i = f^{-1}(\up a_i)$ and $U_i = {\sf Int}\left({\sf Cl}(V_i)\right)$, and set $U_{n+1} = V_{n+1} = \varnothing$. As $f$ is continuous and $\up a_i$ is closed, $V_i$ is closed in $U$. This yields $V_i = {\sf Cl}(V_i) \cap U$. Therefore, $V_i \subseteq U_i \cap U = {\sf Int}\left({\sf Cl}(V_i)\right)\cap U \subseteq {\sf Cl}(V_i) \cap U = V_i$, so $V_i=U_i\cap U$. Let $x \in U$, and suppose that $f(x) = a_i$. Then $x \in V_i - V_{i+1}$, and as $x \in U$, we have $x \in U_i - U_{i+1}$. Thus, $f^\#(x) = a_i = f(x)$, and so $f^\#|_U = f$.

For uniqueness, let $g \in FN(X)$ with $g|_U = f$ and let $a \in A$. Then $g^{-1}(\up a) \cap U = f^{-1}(\up a)$. Since $g^{-1}(\up a)$ is regular open and $U$ is open dense, ${\sf Cl}\left(g^{-1}(\up a)\right) = {\sf Cl}\left(g^{-1}(\up a) \cap U\right)$, so
\[
g^{-1}(\up a) = {\sf Int}\left({\sf Cl}\left(g^{-1}(\up a)\right)\right) = {\sf Int}\left({\sf Cl}\left(g^{-1}(\up a) \cap U\right)\right) = {\sf Int}\left({\sf Cl}\left(f^{-1}(\up a)\right)\right).
\]
This yields $g^{-1}(\up a) = (f^\#)^{-1}(\up a)$ for each $a\in A$, so $g = f^\#$ by Lemma~\ref{refine}.
\end{proof}

The partial order on $F(X)$ restricts to $FN(X)$. By normalizing the join and meet operations on $F(X)$, we obtain operations on $FN(X)$ which we show are the join and meet in $FN(X)$ with respect to the induced partial order on $FN(X)$.

\begin{proposition}\label{prop:2.10}
$FN(X)$ is a lattice, where the meet is the pointwise meet and the join is the normalization of the pointwise join. In other words, if $\wedge,\vee$ denote the meet and join operations on $FN(X)$, then for $f,g\in FN(X)$, we have
\[
f \wedge g = \inf(f,g) \qquad \text{and} \qquad f \vee g = \sup(f,g)^\#.
\]
\end{proposition}

\begin{proof}
First we claim that the normalization operation is order preserving; that is, if $f,g \in F(X)$ with $f \le g$, then $f^\# \le g^\#$. Let $a_0 < \cdots < a_n$ be elements of $A$ containing all values of $f$ and $g$. By Lemma~\ref{refine}, $f^{-1}(\up a_i) \subseteq g^{-1}(\up a_i)$ for each $i$. Therefore, ${\sf Int}\left({\sf Cl}\left(f^{-1}(\up a_i)\right)\right)\subseteq{\sf Int}\left({\sf Cl}\left(g^{-1}(\up a_i)\right)\right)$ for each $i$. Thus, applying Lemma~\ref{refine} again yields $f^\# \le g^\#$.

Now we prove the proposition. Let $f,g \in FN(X)$. Then $\inf(f,g)^{-1}(\up a) = f^{-1}(\up a) \cap g^{-1}(\up a)$ is regular open for each $a \in A$. Thus, $\inf(f,g) \in FN(X)$, and so $f\wedge g=\inf(f,g)$. Since $f \le \sup(f,g)$, we see that $f = f^\# \le \sup(f,g)^\#$. Similarly, $g \le \sup(f,g)^\#$. If $k \in FN(X)$ with $f,g \le k$, then $\sup(f,g) \le k$ in $F(X)$, and so $\sup(f,g)^\# \le k^\# = k$. Consequently, $f\vee g=\sup(f,g)^\#$.
\end{proof}

\section{de Vries powers of totally ordered algebras}

In this section we continue to assume $X$ is a compact Hausdorff space, but we assume now that $A$ is a totally ordered algebra of a given type. We introduce the notion of a de Vries power of $A$ by $X$ (as a set it will be $FN(X)$) in such a way that the power is an algebra of the same type as $A$ and comes equipped with a canonical proximity relation. We first indicate how to lift operations from $A$ to $F(X)$; once this is accomplished, we normalize these operations to obtain an algebraic structure on $FN(X)$ having the same type as that of $A$.

We extend the order and operations on $A$ to $F(X)$ pointwise. That is, for $f,g\in F(X)$, we set $f\le g$ iff $f(x)\le g(x)$ for each $x\in X$, and if $\lambda$ is an $m$-ary operation on $A$ and $f_1,\ldots,f_m \in F(X)$, then we set $\lambda(f_1,\dots, f_m)(x) = \lambda(f_1(x),\dots,f_m(x))$. It is clear that $F(X)$ is a partially ordered algebra of the same type as $A$. Furthermore, if $a \in A$, then
\[
\lambda(f_1,\dots,f_m)^{-1}(a) = \bigcup \{ f_1^{-1}(b_1) \cap \cdots \cap f_m^{-1}(b_n) : \lambda(b_1,\dots, b_m) = a\}.
\]
From this and $FC(X)=C(X,A_{\mathrm{disc}})$ (see Notation~\ref{notation}(2)) it follows that $\lambda(f_1,\dots,f_m) \in FC(X)$ for each $f_1,\dots,f_m \in FC(X)$. Thus, $FC(X)$ is a subalgebra of $F(X)$.

\begin{definition}\label{def:3.1}
For each $m$-ary operation $\lambda$ on $A$, define the $m$-ary operation $\lambda^\#$ on $FN(X)$ by $\lambda^\#(f_1,\dots,f_m) = \lambda(f_1,\dots,f_m)^\#$ for all $f_1,\ldots,f_n \in FN(X)$.
\end{definition}

This makes $FN(X)$ a partially ordered algebra of the same type as $A$, and $FC(X)$ is a subalgebra of $FN(X)$. Alternatively, $FN(X)$ can be viewed as a direct limit of the $FC(U)$, where $U$ ranges over the directed set $\mathcal{I}$ of  dense open subsets of $X$, and the operations on $FN(X)$ then are those induced by the pointwise operations on the $FC(U)$. Theorem~\ref{direct limit theorem} makes this explicit. We use the fact that the direct limit of the directed system $\{FC(U):U \in \mathcal{I}\}$ can be described as the set of all pairs $(f,U)$ with $U \in \mathcal{I}$ and $f \in FC(U)$, and where $(f,U) = (g,V)$ whenever there is $W \in \mathcal{I}$ with $W \subseteq U \cap V$ and $f|_{W} = g|_{W}$ (see \cite[Sec.~1]{BN80}). Since each $FC(U)$ is an algebra of the same type as $A$, the direct limit is also an algebra of the same type as $A$.

\begin{theorem} \label{direct limit theorem}
The algebra $FN(X)$ is isomorphic to the direct limit $L$ of the algebras $FC(U)$ as $U$ ranges over all open dense subsets of $X$.
\end{theorem}

\begin{proof}
For each open dense subset $U$ of $X$, we have a map $FC(U) \to FN(X)$, given by $f \mapsto f^\#$. These then induce a map $\alpha : L\to FN(X)$, given by $\alpha(f,U) = f^\#$. This map is well defined because if $(f,U) = (g,V)$, then $f^\#$ and $g^\#$ are normal functions extending $f$ (and $g$) on a dense open set $W \subseteq U \cap V$. Thus, by Proposition~\ref{uniqueext}, $f^\# = g^\#$. To see that $\alpha$ is a homomorphism, let $\lambda$ be an $m$-ary operation on $A$ and let $g_1,\dots,g_m \in L$. We may find a single open dense set $U$ for which $g_i = (f_i,U)$ for some $f_i \in FC(U)$. Then
\[
\alpha\left(\lambda(g_1,\dots,g_m)\right) = \alpha((\lambda(f_1,\dots, f_m), U)) = \lambda(f_1, \dots, f_m)^\#.
\]
On the other hand,
\[
\lambda^\#(\alpha(g_1), \dots, \alpha(g_m)) = \lambda(f^\#_1, \dots, f^\#_m)^\#.
\]
Both functions $\lambda(f_1,\dots,f_m)^\#$ and $\lambda(f^\#_1,\dots,f^\#_m)^\#$ are normal functions on $X$ and restrict to $\lambda(f_1,\dots,f_m)$ on $U$. Thus, by Proposition~\ref{uniqueext}, they are equal. This proves that $\alpha$ is a homomorphism. It is 1-1 because if $(f,U), (g,V)$ are in $L$ and $f^\# = g^\#$, then by Proposition~\ref{continuousondense}, there is an open dense set $W$ with $f^\#$ continuous on $W$. By replacing $W$ with $W \cap U \cap V$, we may assume $W\subseteq U\cap V$. By Proposition~\ref{restriction}, $f|_W = f^\#|_W = g^\#|_W = g|_W$. Finally, $\alpha$ is onto because if $h \in FN(X)$, then by Proposition~\ref{continuousondense}, $h$ is continuous on an open dense set $U$, so $h = \alpha(h|_U, U)$ by Proposition~\ref{uniqueext}. Consequently, $L$ and $FN(X)$ are isomorphic as algebras.
\end{proof}

We have noted that $FN(X)$ is an algebra of the same type as $A$. The de Vries power of $A$ by $X$ is then the algebra $FN(X)$ equipped with a canonically chosen proximity relation; that the relation indeed behaves like a proximity is proved in Theorem~\ref{prox relation}.

\begin{definition}
\begin{enumerate}
\item[]
\item The {\it de Vries power} of the totally ordered algebra $A$ by $X$ is the algebra  $FN(X)$ with the relation $\prec_X$ defined by
$$
f \prec_X g \mbox{ if } {\sf Cl}\left(f^{-1}(\up a)\right) \subseteq g^{-1}(\up a) \mbox{ for each } a \in A.
$$
In other words, $f \prec_X g$ provided $f^{-1}(\up a)\prec g^{-1}(\up a)$ in the de Vries algebra of regular open subsets of $X$.
\item If $(B,\prec)$ is the de Vries algebra whose de Vries dual is $X$, then we call $(FN(X),\prec_X)$ the {\it de Vries power of $A$ by $(B,\prec)$}.
\end{enumerate}
\end{definition}

As the next theorem shows, $\prec_X$ satisfies  typical axioms for a proximity relation.

\begin{theorem}\label{prox relation}
The relation $\prec_X$ on $FN(X)$ has the following properties.
\begin{enumerate}
\item $f \prec_X g$ implies $f \le g$.
\item $f \le g \prec_X h \le k$ implies $f \prec_X k$.
\item $f \prec_X g,h$ implies $f \prec_X g \wedge h$.
\item $f,g \prec_X h$ implies $f\vee g \prec_X h$.
\item $f \prec_X g$ implies there is $h\in FN(X)$ with $f \prec_X h \prec_X g$.
\item $f \prec_X f$ iff $f \in FC(X)$.
\end{enumerate}
\end{theorem}

\begin{proof}
(1) Let $f \prec_X g$. If $a \in A$, then $f^{-1}(\up a) \prec g^{-1}(\up a)$, so $f^{-1}(\up a) \subseteq g^{-1}(\up a)$. Thus, by Lemma~\ref{refine}, $f \le g$.

(2) Let $f \leq g \prec_X h \leq k$ and let $a \in A$. Then $f^{-1}(\up a) \subseteq g^{-1}(\up a) \prec h^{-1}(\up a) \subseteq k^{-1}(\up a)$, so $f^{-1}(\up a) \prec k^{-1}(\up a)$. Thus, $f \prec_X k$.

(3) Let $f \prec_X g, h$, and let $a \in A$. Then $f^{-1}(\up a) \prec g^{-1}(\up a), h^{-1}(\up a)$, so $f^{-1}(\up a) \prec g^{-1}(\up a) \cap h^{-1}(\up a)$. By Proposition~\ref{prop:2.10}, $g^{-1}(\up a) \cap h^{-1}(\up a) = (g\wedge h)^{-1}(\up a)$. Therefore, $f^{-1}(\up a) \prec (g\wedge h)^{-1}(\up a)$. Thus, $f \prec_X g \wedge h$.

(4) Let $f,g \prec_X h$, and let $a \in A$. Then $f^{-1}(\up a), g^{-1}(\up a) \prec h^{-1}(\up a)$, so $f^{-1}(\up a) \vee g^{-1}(\up a) \prec h^{-1}(\up a).$ Since
\[
{\sf Int}\left({\sf Cl}\left(\sup(f,g)^{-1}(\up a)\right)\right) = {\sf Int}\left({\sf Cl}\left(f^{-1}(\up a) \cup g^{-1}(\up a)\right)\right)  = f^{-1}(\up a) \vee g^{-1}(\up a),
\]
we have ${\sf Int}\left({\sf Cl}\left(\sup(f,g)^{-1}(\up a)\right)\right) \prec h^{-1}(\up a)$. By Proposition~\ref{prop:2.10}, $f\vee g=\sup(f,g)^\#$. Thus, by Definition~\ref{def:2.7}, $(f\vee g)^{-1}(\up a)={\sf Int}\left({\sf Cl}\left(\sup(f,g)^{-1}(\up a)\right)\right)$, so $(f\vee g)^{-1}(\up a) \prec h^{-1}(\up a)$, which implies that $f\vee g \prec_X h$.

(5) Let $f \prec_X g$ and let the values of $f$ and $g$ be among $a_0 < \cdots < a_n$. We have $f^{-1}(\up a_i) \prec g^{-1}(\up a_i)$ for each $i$, so there is a regular open set $U_i$ with $f^{-1}(\up a_i) \prec U_i \prec g^{-1}(\up a_i)$. Set $V_i = U_0 \cap \cdots \cap U_i$ for $0 \le i \le n$ and $V_{n+1} = \varnothing$. Then the $V_i$ are decreasing regular open sets and $f^{-1}(\up a_i) \prec V_i \prec g^{-1}(\up a_i)$ for each $i$. If we define $h$ by $h(x) = a_i$ provided $x \in V_i - V_{i+1}$, then $h\in FN(X)$ and $f \prec_X h \prec_X g$.

(6) Let $a_0 < \cdots < a_n$ be the values of $f$. Then $f \prec_X f$ iff ${\sf Cl}\left(f^{-1}(\up a_i)\right)\subseteq f^{-1}(\up a_i)$ for each $i$, which happens iff $f^{-1}(\up a_i)$ is clopen for each $i$. This is clearly equivalent to $f\in FC(X)$.
\end{proof}

We next show that the notion of a de Vries power of a totally ordered algebra encompasses that of a Boolean power.

\begin{theorem}
If $X$ is a Stone space, then the Boolean power of a totally ordered algebra $A$ by $X$ is the subalgebra $\{f \in FN(X):f \prec_X f\}$ of $FN(X)$. Thus, every Boolean power of a totally ordered algebra can be expressed as a subalgebra of a de Vries power of the algebra.
\end{theorem}

\begin{proof}
As follows from \cite[Prop.~5.4]{BMMO13a}, the Boolean power $C(X,A_\mathrm{disc})$ is equal to $FC(X)$. By Theorem~\ref{prox relation}(6), $FC(X)=\{f \in FN(X):f \prec_X f\}$. As we noted after Definition~\ref{def:3.1}, $FC(X)$ is a subalgebra of $FN(X)$. Therefore, the Boolean power $C(X,A_\mathrm{disc})$ is the subalgebra $\{f \in FN(X):f \prec_X f\}$ of $FN(X)$.
\end{proof}

\begin{corollary} \label{ED}
The Boolean power of a totally ordered algebra $A$ by a Stone space $X$ coincides with the de Vries power of $A$ by $X$ iff $X$ is extremally disconnected.
\end{corollary}

\begin{proof}
It is well known that $X$ is extremally disconnected iff regular opens of $X$ coincide with clopens of $X$. This is clearly equivalent to $FC(X)=FN(X)$. The result follows.
\end{proof}

Of course, the proximity axioms in Theorem~\ref{prox relation} ignore the algebraic structure of $FN(X)$ induced by that of $A$. Such axioms will depend on the behavior of the operations on $A$ with respect to the total order of $A$, and the interplay between the operations and the proximity can be subtle. In what follows, we axiomatize de Vries powers of a totally ordered integral domain, thus generalizing the axiomatization of Boolean powers of a totally ordered domain given in \cite{BMMO13a}. This includes the axiomatization of de Vries powers of such classic algebras as $\mathbb Z$, $\mathbb Q$, and $\mathbb R$, thus generalizing the results of \cite{Rib69, Con74, BMO13a}. It would be of interest to axiomatize de Vries powers of other algebras as well.

\section{de Vries powers of totally ordered domains}

In this section $X$ continues to denote a compact Hausdorff space, but we assume now that $A$ is a totally ordered integral domain. Theorem~\ref{prox relation} indicates how the relation $\prec_X$ on $FN(X)$ behaves with respect to the lattice structure of $FN(X)$. In this section we describe the algebraic structure of $FN(X)$ and show how the ring operations on $FN(X)$ induced by those of the totally ordered domain $A$ interact with the relation $\prec_X$.

Recall that a ring with a partial order $\le$ is an \emph{$\ell$-ring} (lattice-ordered ring) provided (i)~it is a lattice, (ii)~$a \le b$ implies $a + c \le b + c$ for each $c$, and (iii)~$0 \le a, b$ implies $0 \le ab$. An $\ell$-ring is \emph{totally ordered} if the order is a total order, and it is an \emph{$f$-ring} if it is a subdirect product of totally ordered rings. It is well known (see, e.g., \cite[Ch.~XVII, Corollary to Thm.~8]{Bir79}) that an $\ell$-ring is an $f$-ring iff $a\wedge b=0$ and $c\ge 0$ imply $ac\wedge b=0$.

We say a ring $S$ is an {\it $\ell$-algebra} if it is an $\ell$-ring, an $A$-algebra (with $A$ as above), and $a \in A$, $s \in S$ with $0 \le a,s$ imply that $0 \le as$. An $\ell$-algebra $S$ is an {\it $f$-algebra} provided the ring $S$ is an $f$-ring. If $S=\{0\}$, then we call $S$ a {\em trivial} $f$-algebra. If $S$ is a nontrivial torsion-free $f$-algebra, then $a\mapsto a\cdot 1$ embeds $A$ into $S$, and without loss of generality, we view $A$ as a subalgebra of $S$.

\begin{notation}
For a torsion-free $f$-algebra $S$ over $A$, we denote the image $a\cdot 1$ of $a\in A$ in $S$ by $a$. When $S$ is nontrivial, then since $S$ is torsion-free, we may in fact identify $a$ with its image in $S$. However, when $S$ is trivial, then for each $a\in A$, we have $a=0$ in $S$ under our convention. Since we will mostly be dealing with nontrivial algebras, this will cause little confusion.
\end{notation}

\begin{definition}\label{proximity definition}
Let $S$ be a torsion-free $f$-algebra over $A$. A binary relation $\prec$ on $S$ is a {\it proximity} if the following axioms are satisfied:
\begin{enumerate}
\item[(P1)] $0\prec 0$ and $1 \prec 1$.
\item[(P2)] $s \prec t$ implies $s \le t$.
\item[(P3)] $s \le t \prec r \le u$ implies $s \prec u$.
\item[(P4)] $s \prec t,r$ implies $s \prec t \wedge r$.
\item[(P5)] $s \prec t$ implies $-t \prec -s$.
\item[(P6)] $s \prec t$ and $r \prec u$ imply $s+r \prec t+u$.
\item[(P7)] $s \prec t$ implies $as \prec at$ for each $0 < a \in A$, and $as \prec at$ for some $0 < a \in A$ implies $s \prec t$.
\item[(P8)] $s,t,r,u \ge 0$ with $s \prec t$ and $r \prec u$ imply $sr \prec tu$.
\item[(P9)] $s \prec t$ implies there is $r\in S$ with $s \prec r \prec t$.
\item[(P10)] $s > 0$ implies there is $0 < t\in S$ with $t \prec s$.
\end{enumerate}
A pair $(S,\prec)$ is a {\it proximity $A$-algebra} if $S$ is a torsion-free $f$-algebra over $A$ and $\prec$ is a proximity on $S$. If $S$ is a Specker $A$-algebra, then we call $(S,\prec)$ a {\it proximity Specker $A$-algebra}.
\end{definition}

\begin{remark}\label{consequences of proximity definition}
\begin{enumerate}
\item[]
\item It is an easy consequence of the axioms that $s\prec t$ and $r\prec u$ imply $s\wedge r\prec t\wedge u$ and $s\vee r\prec t\vee u$. Also, it follows from (P1), (P7), and (P5) that for each $a \in A$, we have $a \prec a$.
\item In ``good" cases, one implication of axiom (P7), that $as \prec at$ for some $0 < a \in A$ implies $s\prec t$, is superfluous. For example, if $A$ is a field and $as\prec at$ for some $0<a\in A$, then by the other implication of axiom (P7), we obtain $a^{-1}as\prec a^{-1}at$, so $s\prec t$. It is also superfluous in some other cases, but we leave the details out because in what follows we will use axiom (P7) in its full strength.
\item By \cite[Thm.~5.1]{BMMO13a}, each Specker $A$-algebra has a unique partial order making it into an $f$-algebra. Since $A$ is a domain, it is a torsion-free $f$-algebra, so proximity Specker $A$-algebras are well defined.
\end{enumerate}
\end{remark}

We show in Theorem~\ref{FN algebra} that not only is $FN(X)$ a proximity $A$-algebra, but it has the particularly transparent algebraic structure of a Baer  Specker $A$-algebra, a notion we recall now.

\begin{definition}
A \emph{Baer ring} $S$ is a commutative ring such that the annihilator of each subset of $S$ is generated as an ideal by an idempotent of the ring (see,~e.g., \cite[p.~260]{Lam99}). A Specker algebra $S$ over the domain $A$ is a \emph{Baer Specker $A$-algebra} provided $S$ is a Baer ring, and a proximity Specker $A$-algebra $(S,\prec)$ is a \emph{proximity Baer Specker $A$-algebra} provided $S$ is a Baer Specker $A$-algebra.
\end{definition}

A Specker $A$-algebra is a Baer Specker $A$-algebra iff the Boolean algebra $\Id(S)$ of idempotents of $S$ is a complete Boolean algebra \cite[Thm.~4.3]{BMMO13a}. To prove that $FN(X)$ is a proximity Baer Specker $A$-algebra, we rely on the following lemma, which can be viewed as a description of the operations on $FN(X)$ that are lifted from those of the domain $A$.

We note that the operations on $A$ are lifted to $F(X)$ pointwise, while the operations on $FN(X)$ are normalizations of the operations on $F(X)$. In particular, $FN(X)$ is not a subalgebra of $F(X)$. For $f,g\in FN(X)$, we denote the sum, product, and join of $f$ and $g$ in $FN(X)$ by $f+g$, $fg$, and $f\vee g$, respectively. So $f+g$ is the normalization of the pointwise sum, $fg$ is the normalization of the pointwise product, and $f\vee g$ is the normalization of the pointwise join of $f$ and $g$. On the other hand, as was shown in Proposition~\ref{prop:2.10}, $f\wedge g$ is the pointwise meet of $f$ and $g$.

\begin{lemma}\label{AlgebraicPropertiesOfFN}
Let $f,g \in FN(X)$ and let $a \in A$.
\begin{enumerate}
\item[(1)] $(f+g)^{-1}(\up a) = \bigvee\{f^{-1}(\up b) \cap g^{-1}(\up c):b+c \ge a\}$.
\item[(2)] If $f,g \ge 0$, then $(fg)^{-1}(\up a) = \bigvee\{f^{-1}(\up b) \cap g^{-1}(\up c):b,c\ge 0, bc\ge a\}$.
\item[(3)] If $0 < c \in A$, then $(cf)^{-1}(\up a) = f^{-1}(\up b)$, where $b$ is the smallest value of $f$ for which $bc \ge a$. Furthermore, $cf$ is the pointwise scalar product of $c$ and $f$.
\item[(4)] If $0 < c \in A$, then $(cf)^{-1}(\up(ca)) = f^{-1}(\up a)$.
\item[(5)] $(-f)^{-1}(\up a) = \lnot f^{-1}(\up b)$, where $b$ is the smallest value of $f$ larger than $-a$ (provided it exists, otherwise $b$ is any element of $A$ larger than $-a$).
\end{enumerate}
\end{lemma}

\begin{proof}
(1) Let $h$ be the pointwise sum of $f$ and $g$. Then $f+g$ is the normalization of $h$, and $f+g$ is determined by the formula
\[
(f+g)^{-1}(\up a) = {\sf Int}\left({\sf Cl}\left(h^{-1}(\up a)\right)\right)
\]
for all $a \in A$. We claim that
\[
(f+g)^{-1}(\up a) = \bigvee\{f^{-1}(\up b) \cap g^{-1}(\up c):b+c \ge a\},
\]
where the join is in $\mathcal{RO}(X)$. To see this, we first point out that
\[
h^{-1}(\up a) = \bigcup\{f^{-1}(\up b) \cap g^{-1}(\up c):b+c \ge a\}.
\]
Therefore,
\begin{eqnarray*}
(f+g)^{-1}(\up a) &=& {\sf Int}\left({\sf Cl}\left(h^{-1}(\up a)\right)\right) \\
&=& {\sf Int}\left({\sf Cl}\left(\bigcup\{f^{-1}(\up b) \cap g^{-1}(\up c):b+c \ge a\}\right)\right) \\
&=& \bigvee\{f^{-1}(\up b) \cap g^{-1}(\up c):b+c \ge a\}.
\end{eqnarray*}

(2) Suppose that $f,g \ge 0$. Let $h$ be the pointwise product of $f$ and $g$. Then $fg$ is the normalization of $h$. Since $f,g\ge 0$,
\[
h^{-1}(\up a) = \bigcup\{f^{-1}(\up b) \cap g^{-1}(\up c):b,c \ge 0, bc \ge a\}.
\]
Therefore,
\begin{eqnarray*}
(fg)^{-1}(\up a) &=& {\sf Int}\left({\sf Cl}\left(\bigcup\{f^{-1}(\up b) \cap g^{-1}(\up c):b,c \ge 0, bc \ge a\}\right)\right) \\
&=& \bigvee\{f^{-1}(\up b) \cap g^{-1}(\up c):b,c \ge 0, bc \ge a\}.
\end{eqnarray*}

(3) Let $h$ be the pointwise scalar product of $c$ and $f$. Then $cf$ is the normalization of $h$. Observe that
\[
h^{-1}(\up a) = \bigcup\{f^{-1}(\up b):cb \ge a\}.
\]
Because $f$ is finitely valued, there is a smallest value $b$ of $f$ for which $cb \ge a$. The formula above then shows that $h^{-1}(\up a) = f^{-1}(\up b)$. This implies that $h^{-1}(\up a)$ is regular open, so $h$ is normal. Thus, $cf = h$ is the pointwise scalar product of $c$ and $f$.

(4) This follows from (3) since $A$ being a totally ordered domain and $0 < c$ imply $ca \le cb$ iff $a \le b$.

(5) Let $h$ be the pointwise negative of $f$ (that is, $h(x) = -f(x)$ for each $x \in X$), and let $-f$ be the normalization of $h$. Since $f$ is finitely valued, we have
\[
h^{-1}(\up a) = \{ x \in X : a \le h(x) \} = \{ x \in X : f(x) \le -a\} = X - f^{-1}(\up b),
\]
where $b$ is the smallest value of $f$ larger than $-a$ (provided it exists, otherwise $b$ can be any element of $A$ larger than $-a$). Therefore,
\[
(-f)^{-1}(\up a) = {\sf Int}\left({\sf Cl}\left(h^{-1}(\up a)\right)\right) = {\sf Int}\left({\sf Cl}\left(X -f^{-1}(\up b)\right)\right) = {\sf Int}\left(X - f^{-1}(\up b)\right) = \lnot f^{-1}(\up b).
\]
\end{proof}

\begin{remark}\label{pointwise}
The operations of addition, multiplication, and join in $FN(X)$ are in general not pointwise. In fact, any one of these operations is pointwise iff $X$ is extremally disconnected. One implication follows from Corollary~\ref{ED}. To see the converse, say for addition, let $U$ be a regular open subset of $X$ that is not clopen. Then the pointwise sum of $\chi_U$ and $\chi_{\lnot U}$ is $\chi_{U \cup \lnot U}$, and since $U \cup \lnot U\ne X$, we see that $\chi_{U \cup \lnot U}\ne 1$. On the other hand, since $U \cup \lnot U$ is dense in $X$, we have $\chi_U + \chi_{\lnot U}= (\chi_{U \cup \lnot U})^\# = \chi_X = 1$ in $FN(X)$.
\end{remark}

\begin{remark}\label{special case}
By Remark~\ref{pointwise}, for $f,g \in FN(X)$, the sum $f+g$ in $FN(X)$ need not be the pointwise sum. In spite of this, if one of $f,g$ is a constant function, then $f+g$ is the pointwise sum. The same is true for join and multiplication by a positive scalar. That scalar multiplication by a positive scalar is pointwise was pointed out in Lemma~\ref{AlgebraicPropertiesOfFN}(3). More generally, these facts are immediate consequences of the following facts about normalization. Let $f \in F(X)$ and $a \in A$. For notational convenience, let $+$ refer to the pointwise sum in $F(X)$. Then
\begin{eqnarray*}
(a+f)^\# &=& a + f^\#, \\
\sup(a,f)^\# &=& a \vee f^\#.
\end{eqnarray*}
In addition, if $0 \le a$ and $\cdot$ refers to pointwise multiplication, then
\begin{eqnarray*}
(a\cdot f)^\# = a\cdot f^\#.
\end{eqnarray*}
The arguments for each of these statements are similar, so we give the proof for the first. For each $b \in A$ we have
\begin{eqnarray*}
[(a+ f)^\#]^{-1}(\up b) &=& {\sf Int}\left({\sf Cl}\left((a+ f)^{-1}(\up b)\right)\right) = {\sf Int}\left({\sf Cl}\left(f^{-1}(\up (b-a))\right)\right) \\
&=& (f^\#)^{-1}({\up}(b-a)) = (a + f^\#)^{-1}(\up b).
\end{eqnarray*}
Thus, $(a+ f)^\# = a + f^\#$.

In contrast, scalar multiplication by a negative scalar is not pointwise. To see this, let $U$ be regular open that is not clopen, and let $f=\chi_U$. Then $(-1)f = (-1)\chi_U = -1 + \chi_{X-U}$, where $-1 + \chi_{X-U}$ is the pointwise sum. Therefore, $(-1)f$ is not normal because $\left((-1)f\right)^{-1}(\up 0)=X-U$ is not regular open.

Finally, we point out that by Remark~\ref{decreasing}, if $a_0 < \cdots < a_n$ and $X = U_0 \supset U_1 \supset \cdots \supset U_n \supset \varnothing$ are regular open, then the pointwise sum $a_0 + \sum_{i=1}^n (a_i - a_{i-1}) \chi_{U_i}$ is normal. Thus, the sum of $a_0$ and the $(a_i - a_{i-1})\chi_{U_i}$ in $FN(X)$ is the same as the sum in $F(X)$.
\end{remark}

It is well known that in a commutative ring $S$, the set $\Id(S)$ of idempotents of $S$ forms a Boolean algebra with respect to the operations $s\wedge t=st, \ s\vee t=s+t-st, \ \lnot s=1-s$.

\begin{lemma}\label{idempotents of FN}
$FN(X)$ is a commutative ring with 1, and $f\in FN(X)$ is an idempotent iff $f=\chi_U$ for some regular open $U$. Moreover, the map $U \mapsto \chi_U$ is a Boolean isomorphism between $\mathcal{RO}(X)$ and $\Id(FN(X))$.
\end{lemma}

\begin{proof}
Observe that $FC(U)$ is a commutative ring with 1 for each open dense subset $U$ of $X$. Therefore, so is the direct limit of the $FC(U)$. Now apply Theorem~\ref{direct limit theorem} to conclude that $FN(X)$ is a commutative ring with 1.

Let $U \in \mathcal{RO}(X)$. Then $\chi_U$ is idempotent in $F(X)$, and since it is a normal function, it is idempotent in $FN(X)$. Conversely, suppose that $f \in \Id(FN(X))$. If $h$ is the pointwise square of $f$, then $f^2 = h^\#$. By Remark~\ref{immediate facts}(2), $\func{Im}(h^\#) \subseteq \func{Im}(h)$. This implies that $\func{Im}(f) = \func{Im}(f^2) \subseteq \func{Im}(h)$. Moreover, $\func{Im}(h) = \{a^2 : a \in \func{Im}(f)\}$. Let $a_0 < \cdots < a_n$ be the values of $f$. The values of $h$ are $a_0^2, \cdots, a_n^2$. Because $\func{Im}(f) \subseteq \func{Im}(h)$, we have $\func{Im}(f) = \func{Im}(h)$. This then yields $a_i^2 = a_i$ for each $i$. Thus, $a_i \in \{0,1\}$. Consequently, if $U = f^{-1}(1)$, then $f = \chi_U$. By Remark~\ref{immediate facts}(2) and the fact that $f$ is normal, we obtain $\chi_U = f = f^\# = \chi_{{\sf Int}\left({\sf Cl}(U)\right)}$, which shows that $U$ is regular open in $X$.

It is clear that $\chi_{U\cap V}=\chi_U\wedge\chi_V$. Also, if $h$ is the pointwise negation of $\chi_U$, then in $F(X)$ we have $h=-1+\chi_{X-U}$. Therefore, by Remark~\ref{special case}, in $FN(X)$ we have $-\chi_U = h^\# = -1+\chi_{{\sf Int}(X-U)} = -1 + \chi_{\lnot U}$. Thus, in $FN(X)$ we have
\[
\lnot \chi_U = 1-\chi_U  = 1 + (-1 + \chi_{\lnot U}) = \chi_{\lnot U}.
\]
This yields that $U \mapsto \chi_U$ is a Boolean isomorphism between $\mathcal{RO}(X)$ and $\Id(FN(X))$.
\end{proof}

In order to prove that de Vries powers of $A$ are proximity Baer Specker $A$-algebras, we need the following lemma, which will also be used in later sections. We recall \cite[Thm.~5.1]{BMMO13a} that a Specker $A$-algebra $S$ has a unique partial order $\le$ for which $S$ is a torsion-free $f$-algebra over $A$.

\begin{lemma}\label{lem:3.9}
Let $S$ be a torsion-free $f$-algebra over $A$.
\begin{enumerate}
\item If $e \in \func{Id}(S)$, then $0 \le e \le 1$.
\item The restriction of $\le$ to $\Id(S)$ is the Boolean ordering on $\Id(S)$.
\item If $e\in\func{Id}(S)$ and $a\in A$ with $0 \le a \le 1$, then $a \wedge e = ae$.
\item If $e\in\func{Id}(S)$ and $a\in A$ with $1 \le a$, then $ae \wedge 1 = e$.
\item If $0\ne e\in\func{Id}(S)$ and $a\in A$ with $ae\ge 0$, then $a\ge 0$.
\item Let $0\ne e,k \in \func{Id}(S)$ and $0< a,b \in A$. Then $ae \le bk$ iff $a \le b$ and $e \le k$.
\end{enumerate}
\end{lemma}

\begin{proof}
(1) Let $e \in \func{Id}(S)$. Then $e = e^2$ is a square in $S$. Since $S$ is an $f$-ring, squares in $S$ are nonnegative \cite[Ch.~XVII, Lem.~2]{Bir79}. This forces $e \ge 0$. The same argument shows that $1-e \ge 0$, so $e \le 1$.

(2) Let $e,k \in \func{Id}(S)$. We must show that $e \le k$ iff $ek=e$. If $ek = e$, then by (1), $e = ek \le 1\cdot k = k$. Conversely, suppose that $e \le k$. By (1), $0 \le 1-k$. Therefore, $0 \le e(1-k) \le k (1-k)= 0$. Thus, $e(1-k)=0$, so $ek =e$.

(3) Let $e \in \func{Id}(S)$ and $a\in A$ with $0 \le a \le 1$. We have $e \wedge (1-e) = 0$. Then, since $1-a\ge 0$ and $S$ is an $f$-ring, $(1-a)e \wedge (1-e) = 0$. Using again that $S$ is an $f$-ring and $a\ge 0$, we obtain $(1-a)e \wedge a(1-e) = 0$. Therefore, $(e-ae) \wedge (a-ae) = 0$. As $S$ is an $\ell$-ring, we have $(r+t)\wedge(s+t)=(r\wedge s)+t$ for each $r,s,t\in S$. This implies $(e\wedge a)-ae=0$, so $a \wedge e = ae$.

(4) Let $e\in\func{Id}(S)$ and $a\in A$ with $1 \le a$. We have $(ae \wedge 1) - e = (a-1)e \wedge (1-e)$. Now, since $e \wedge (1-e) = 0$, $a-1\ge 0$, and $S$ is an $f$-ring, $(a-1)e \wedge (1-e) = 0$. Thus, $(ae\wedge 1)-e=0$, so $ae \wedge 1 = e$.

(5) Let $0\ne e\in\func{Id}(S)$ and $a\in A$ with $ae\ge 0$. By (1), $e\ge 0$. If $a\not\ge 0$, then as $A$ is totally ordered, $a<0$. Therefore, $-a>0$, so $-ae\ge 0$. Thus, since $ae, -ae \ge 0$, we see that $ae=0$. As $e\ne 0$ and $S$ is torsion-free, we conclude that $a=0$, a contradiction. Consequently, $a\ge 0$.

(6) One implication is obvious. For the other, suppose that $ae \le bk$. Then $0\le ae(1-k)\le bk(1-k)=0$. Therefore, $ae(1-k)=0$. As $a\ne 0$ and $S$ is torsion-free, $e(1-k)=0$, so $e=ek$. This, by (2), implies that $e\le k$. Next, $ae\le bk$ implies $ae^2\le bek$, so $ae\le be$. Therefore, $(b-a)e\ge 0$. Thus, by (5), $b-a\ge 0$, so $a\le b$.
\end{proof}

\begin{theorem} \label{FN algebra}
The de Vries power $(FN(X), \prec_X)$ of $A$ is a proximity Baer Specker $A$-algebra.
\end{theorem}

\begin{proof}
We first show that $FN(X)$ is a Baer Specker $f$-algebra. If $f \in FN(X)$, then Lemma~\ref{easy first}(1) and Remark~\ref{decreasing} show that $f$ is a linear combination of idempotents. To see that $FN(X)$ is torsion-free over $A$, if $f \in FN(X)$ and $0 \ne a \in A$ with $af = 0$, we may assume that $a > 0$. By Lemma~\ref{AlgebraicPropertiesOfFN}(3), $af$ is the pointwise scalar product. Therefore, $af(x) = 0$ for each $x \in X$. Since $A$ is a domain, this forces $f(x) = 0$ for all $x\in X$. Thus, $f = 0$, and so $FN(X)$ is a Specker $A$-algebra. Next, by Lemma~\ref{idempotents of FN}, $\Id(FN(X))$ is isomorphic to the complete Boolean algebra $\mathcal{RO}(X)$, so $FN(X)$ is a Baer ring by \cite[Thm.~4.3]{BMMO13a}. Finally, to see that $FN(X)$ is an $f$-algebra with respect to the pointwise order $\le$, since $FN(X)$ is a Specker $A$-algebra, by \cite[Thm.~5.1]{BMMO13a}, it has a unique partial order $\le^\prime$ which makes it into an $f$-algebra. We show that $\le^\prime$ is the same as the pointwise order $\le$ on $FN(X)$. Since $FN(X)$ is an $f$-algebra with respect to $\le^\prime$, squares are nonnegative \cite[Ch.~XVII, Lem.~2]{Bir79}. Therefore, idempotents in $FN(X)$ are nonnegative. Let $f \in FN(X)$ and let $a_0 < \cdots < a_n$ be the values of $f$. Set $U_i = f^{-1}({\uparrow}a_i)$ for each $i$. Then each $U_i$ is regular open and $f = a_0 + \sum_{i=1}^{n} (a_{i} - a_{i-1}) \chi_{U_{i}}$. Since $a_i - a_{i-1} > 0$ and $\chi_{U_i} \ge^\prime 0$, we have $0 \le^\prime \sum_{i=1}^{n} (a_{i} - a_{i-1}) \chi_{U_{i}}$. Thus, if $0 \le a_0$, then $0 \le^\prime f$. Conversely, suppose that $0 \le^\prime f$. If $f = a_0$, then $0\le a_0$. Suppose that $n \ge 1$. Then $U_1$ is a proper regular open set. Therefore, $\lnot U_1 \ne \varnothing$, so $\chi_{\lnot U_1}$ is a nonzero idempotent in $FN(X)$. Since it is nonnegative, we get $0 \le^\prime f\chi_{\lnot U_1} = a_0\chi_{\lnot U_1}$. Consequently, by Lemma~\ref{lem:3.9}(5), $a_0 \ge 0$. Thus, $0 \le^\prime f$ iff $0 \le a_0$. On the other hand, it is clear for the pointwise order $\le$ that $0 \le f$ iff $0 \le a_0$ as $a_0$ is the smallest value of $f$. Thus, $\le^\prime$ is equal to $\le$, and so $FN(X)$ is an $f$-algebra with respect to the pointwise order $\le$.

It remains to show that $\prec_X$ is a proximity in the sense of Definition~\ref{proximity definition}. That axioms (P1), (P2), (P3), (P4), and (P9) hold follows from Theorem~\ref{prox relation}.

(P5) Suppose that $f \prec_X g$. Then $f^{-1}(\up a) \prec g^{-1}(\up a)$ for each $a \in A$, so $\lnot g^{-1}(\up a) \prec \lnot f^{-1}(\up a)$. Therefore, by Lemma~\ref{AlgebraicPropertiesOfFN}(5), $(-g)^{-1}(\up a) \prec (-f)^{-1}(\up a)$. Thus, $-g \prec_X -f$.

(P6) Suppose that $f \prec_X h$ and $g \prec_X k$. Then, for each $b \in A$, we have $f^{-1}(\up b) \prec h^{-1}(\up b)$ and $g^{-1}(\up b) \prec k^{-1}(\up b)$. Therefore, by Lemma~\ref{AlgebraicPropertiesOfFN}(1) and the fact that the join in question involves only finitely many regular open sets, $(f+g)^{-1}(\up a) \prec (h+k)^{-1}(\up a)$ for each $a \in A$. Thus, $f+g \prec_X h+k$.

(P7) Suppose that $f \prec_X g$ and $0 < c \in A$. Lemma~\ref{AlgebraicPropertiesOfFN}(3) then implies that $(cf)^{-1}(\up a) \prec (cg)^{-1}(\up a)$ for each $a \in A$. Thus, $cf \prec_X cg$. Conversely, suppose that $c > 0$ and $cf \prec cg$. Then, for each $b \in A$, we have $(cf)^{-1}(\up b) \prec (cg)^{-1}(\up b)$. Setting $b = ca$ and applying Lemma~\ref{AlgebraicPropertiesOfFN}(4) yields $f^{-1}(\up a) \prec g^{-1}(\up a)$ for each $a \in A$. Thus, $f \prec_X g$.

(P8) Suppose that $f,g,h,k\ge 0$, $f \prec_X h$, and $g \prec_X k$. Lemma~\ref{AlgebraicPropertiesOfFN}(2) and the fact that the join in question involves only finitely many regular open sets then give $(fg)^{-1}(\up a) \prec (hk)^{-1}(\up a)$ for each $a \in A$. Thus, $fg \prec_X hk$.

(P10) Let $0 < g$. Then $X = g^{-1}(\up 0)$ and there is $b > 0$ with $g^{-1}(\up b) \ne \varnothing$. Let $a_0 < \cdots < a_n$ be the values of $g$. We have $0 \le a_0$ and at least one of the $a_i$ satisfies $a_i > 0$. For each $i > 1$ choose a regular open $U_i$ with $\varnothing \ne U_i \prec g^{-1}(\up a_i)$, and set $U_0 = X$. Since $g^{-1}(\up a_0) = X$, we have $U_i \prec g^{-1}(\up a_i)$ for each $i$. Set $V_i = U_0 \cap \cdots \cap U_i$ for $0 \le i \le n$ and $V_{n+1} = \varnothing$. Define $f$ by $f(x) = a_i$ if $x \in V_i - V_{i+1}$. Then $f \in FN(X)$ and $f \prec_X g$. Furthermore, $0 < f$ since each value of $f$ is at least $0$ and one value is greater than $0$.
\end{proof}

In Corollary~\ref{FN algebra char} we prove the converse, that every proximity Baer Specker $A$-algebra is of the form $(FN(X),\prec_X)$ for an appropriate choice of $X$. This is accomplished through a more nuanced investigation of proximity Specker $A$-algebras.

\section{Proximity Specker algebras}

In this section we continue to assume that $A$ is a totally ordered domain; however, we no longer assume that $X$ is a fixed compact Hausdorff space.
In the last section, we established that when $X$ is a compact Hausdorff space, then $(FN(X),\prec_X)$ is a proximity Baer Specker $A$-algebra. In this section we show that every proximity Baer Specker $A$-algebra is of the form $(FN(X),\prec_X)$, for some compact Hausdorff space $X$, and we prove a uniqueness statement for the proximity on $FN(X)$. In fact, these results can be framed in the more general context of proximity Specker $A$-algebras.

Let $S$ be a Specker $A$-algebra. We call a set $E$ of nonzero idempotents of $S$ \emph{orthogonal} if $ek=0$ for all $e\ne k$ in $E$. We say that $s\in S$ is in \emph{orthogonal form} provided $s = \sum_{i=0}^n a_i e_i$, where the $a_i \in A$ are distinct and the $e_i$ are orthogonal. If in addition $\bigvee_{i=0}^n e_i=1$, then we say that $s$ is in \emph{full orthogonal form}. By \cite[Lem.~2.1]{BMMO13a}, each $s \in S$ has a unique full orthogonal decomposition.

We say that $s \in S$ is in \emph{decreasing form} if $s= a_0 + \sum_{i=1}^n b_ik_i$, where each $b_i > 0$ and $1 = k_0 > k_1 > \cdots > k_n > 0$. There is a close connection between orthogonal and decreasing decompositions. To see this, write $s = \sum_{i=0}^n a_ie_i$ in full orthogonal form, and suppose $a_0 < \cdots < a_n$. Since the $e_i$ are orthogonal, $e_i + \cdots + e_n = e_i \vee \cdots \vee e_n$. Therefore,
\begin{eqnarray*}
s &=& \sum_{i=0}^n a_i e_i = a_0(e_0 + \cdots + e_n) + (a_1 - a_0)(e_1 + \cdots + e_n) + \cdots + (a_n - a_{n-1})e_n \\
&=& a_0 + \sum_{i=1}^{n} (a_{i} - a_{i-1}) (e_{i} \vee \cdots \vee e_n) = a_0 + \sum_{i=1}^{n} (a_{i} - a_{i-1}) k_{i},
\end{eqnarray*}
where $k_{i}=e_{i} \vee \cdots \vee e_n$. This writes $s$ in decreasing form. Conversely, if $s= a_0 + \sum_{i=1}^n b_ik_i$ is in decreasing form, then we can recover the full orthogonal decomposition of $s$ as follows:
\begin{eqnarray*}
s &=& a_0 + b_{1}(k_1-k_2) + (b_1 + b_2)(k_2-k_3) + \cdots + (b_1 + \cdots + b_{n-1})(k_{n-1}- k_n) \\
& & + (b_1 + \cdots + b_n)k_n \\
&=& a_0(k_0 \wedge \lnot k_1) +  (a_0 + b_{1})(k_1 \wedge \lnot k_2) + (a_0 + b_1 + b_2)(k_2 \wedge \lnot k_3) + \cdots \\
& & + \left(a_0 + \sum_{i=1}^{n-1} b_i\right)(k_{n-1} \wedge \lnot k_{n}) + \left(a_0 + \sum_{i=1}^n b_n\right)k_{n}.
\end{eqnarray*}
Set $e_i=k_i\wedge \lnot k_{i+1}$ and for $i \ge 1$, set $a_{i} = a_0 + \sum_{j=1}^i b_{j}$. Then $a_0 < \cdots < a_n$ and $s=\sum_{i=0}^n a_ie_i$ is in full orthogonal form. Since $s$ has a unique full orthogonal decomposition, we see from the above correspondence that $s$ also has a unique decreasing decomposition.

\begin{proposition}\label{prop:4.6}
If $(S,\prec)$ is a proximity Specker $A$-algebra, then $\prec$ restricts to a proximity on $\Id(S)$.
\end{proposition}

\begin{proof}
Axioms (DV1)--(DV4) are obvious. To verify (DV5), let $e,k\in\func{Id}(S)$ with $e\prec k$. By (P5), $-k\prec-e$, so (P1) and (P6) yield $1-k\prec 1-e$. But $1-k=\neg k$ and $1-e=\neg e$. Thus, $\neg k\prec\neg e$, and (DV5) is satisfied. To verify (DV6), let $e,k \in \func{Id}(S)$ with $e \prec k$. By (P9), there is $s \in S$ with $e \prec s \prec k$. Write $s = \sum_{i=1}^n a_i e_i$ in orthogonal form with each $a_i \ne 0$. Since the $e_i$ are orthogonal, $se_i=a_ie_i$ for each $i$. By Lemma~\ref{lem:3.9}(1), $e_i\ge 0$, so as $s\ge 0$, we have $se_i\ge 0$, hence $a_ie_i\ge 0$ for each $i$. This, by Lemma~\ref{lem:3.9}(5), yields $a_i > 0$. Since $a_i e_i \le s$ for each $i$, by (P3), $a_i e_i \prec k$. Then $a_i e_i \le k$ by (P2), so $a_i \le 1$ by Lemma~\ref{lem:3.9}(6). From Lemma~\ref{lem:3.9}(1), (P1), and (P3) it follows that $e_i \prec 1$. Therefore, by (P7), $a_i e_i \prec a_i$. Thus, by (P4) and Lemma~\ref{lem:3.9}(3), $a_ie_i \prec a_i \wedge k = a_ik$. Then (P7) yields $e_i \prec k$. This implies $l := e_1 \vee \cdots \vee e_n \prec k$. Finally, if $a = \sum_{i=1}^n a_i$, then $s \le a(e_1 \vee \cdots \vee e_n) = al$, so $e\prec al$ by (P3). Then $e \le al$ by (P2), so $1 \le a$ by Lemma~\ref{lem:3.9}(6). Therefore, as $e \prec 1$ by Lemma~\ref{lem:3.9}(1), (P1), and (P3), we obtain $e \prec al\wedge 1 = l$ by (P4) and Lemma~\ref{lem:3.9}(4). We thus have an idempotent $l$ with $e \prec l \prec k$, so (DV6) is satisfied. To verify (DV7), let $k$ be a nonzero idempotent in $S$. By (P11), there is $1 < s \in S$ with $s \prec k$. Write $s = \sum_{i=1}^n a_i e_i$ as before. Then $a_1 e_1 \le s$, so $a_1 e_1 \prec k$ by (P3), and the same argument as above yields $e_1 \prec k$. Therefore, (DV7) is satisfied. Thus, the restriction of $\prec$ to $\Id(S)$ is a proximity on $\Id(S)$.
\end{proof}

We next use Proposition~\ref{prop:4.6} to establish a representation theorem for an arbitrary proximity Specker $A$-algebra $(S,\prec)$ by showing that $(S,\prec)$ embeds into $(FN(X),\prec_X)$ for an appropriate choice of compact Hausdorff space $X$. Specifically, $X$ is the space of ends of $(\Id(S),\prec)$, which we next recall. Let $B$ be a Boolean algebra and let $\prec$ be a proximity on $B$. For $E\subseteq B$, let ${\thd}E = \{a\in B: a\prec e \mbox{ for some } e\in E\}$, and define ${\thu}E$ dually. We call an ideal $I$ of $B$ \emph{round} if $I={\thd}I$. Dually, we call a filter $F$ of $B$ \emph{round} if ${\thu}F=F$. The dual compact Hausdorff space of $(B,\prec)$ can be constructed either by means of maximal round ideals or maximal round filters of $(B,\prec)$. In fact, there is a bijection between maximal round filters and maximal round ideals given by $F \mapsto \{ b : \lnot b \in F \}$. De Vries preferred to work with maximal round filters. We will instead work with maximal round ideals. Our choice is motivated by their close connection to minimal prime ideals of a Specker $A$-algebra, which will be discussed in Section 7. Because maximal round filters are called ends in the literature, we will use the same term for maximal round ideals.

Let $X$ be the set of ends of $(B,\prec)$. For $a\in B$, let $\zeta(a)=\{x\in X: a\in x\}$. Define a topology on $X$ by letting $\zeta[B]=\{\zeta(a) : a\in B\}$ be a basis for the topology. The bijection above is a homeomorphism between $X$ and the space of maximal round filters, topologized by the basis consisting of $\xi(a) = \{F : a \in F\}$ for $a \in B$.  By \cite[Ch.~I.3]{deV62}, the space of maximal round filters is compact Hausdorff. Thus, $X$ is compact Hausdorff.

Adopting \cite[Def.~I.3.7]{deV62}, we call a subset $T$ of a proximity Specker $A$-algebra $(S,\prec)$ \emph{dense} if for each $s,r\in S$ with $s\prec r$, there is $t\in T$ with $s\prec t\prec r$.

\begin{theorem}\label{rep}
Let $(S,\prec)$ be a proximity Specker $A$-algebra, and let $X$ be the space of ends of $(\Id(S),\prec)$. Then there is an $\ell$-algebra embedding $\eta : S \rightarrow FN(X)$ such that $\eta[S]$ is dense in $FN(X)$ and $s \prec t$ iff $\eta(s) \prec_X \eta(t)$ for all $s,t \in S$.
\end{theorem}

\begin{proof}
Let $B = \func{Id}(S)$. By Proposition~\ref{prop:4.6}, the restriction of $\prec$ is a proximity on $B$. Let $X$ be the space of ends of $(B,\prec)$. Then $\zeta : B \to \mathcal{RO}(X)$ is an embedding \cite[Ch.~I.3]{deV62}. We thus have a map $\sigma : B \to \func{Id}(FN(X))$ defined by $\sigma(e) = \chi_{\zeta(e)}$. This is a Boolean homomorphism since
\[
\sigma(e\wedge k)=\chi_{\zeta(e\wedge k)}=\chi_{\zeta(e) \cap \zeta(k)}=\chi_{\zeta(e)} \wedge \chi_{\zeta(k)}=\sigma(e)\wedge\sigma(k)
\]
and
\[
\sigma(\lnot e)=\chi_{\zeta(\lnot e)}=\chi_{\lnot\zeta(e)}=\lnot\chi_{\zeta(e)}=\lnot\sigma(e).
\]
Since $S$ is a Specker $A$-algebra, by \cite[Sec.~2]{BMMO13a} there is a uniquely determined $A$-algebra homomorphism $\eta : S \to FN(X)$ extending $\sigma$. By \cite[Cor.~5.3]{BMMO13a}, $\eta$ is an $\ell$-algebra homomorphism. To see that $\eta$ is 1-1, let $s\ne 0$. As noted in the beginning of the section, we may write $s$ in decreasing form
\[
s = a_0 + \sum_{i=1}^{n} (a_{i} - a_{i-1}) e_{i},
\]
with $a_0 < \cdots < a_n$ in $A$ and $1 = e_0 > e_1 > \cdots > e_n>0$ in $B$. Then $\eta(s) = a_0 + \sum_{i=1}^{n} (a_{i} - a_{i-1}) \chi_{\zeta(e_{i})}$. Therefore, $\eta(s)(x) = a_i$ provided $x \in \zeta(e_i) - \zeta(e_{i+1})$. If $s = a_0$, then as $s \ne 0$, we have $a_0\ne 0$, so $\eta(s)\ne 0$. Otherwise $n > 0$, so $e_1 \ne 0$. Thus, if $x \in \zeta(e_1)$, then $\eta(s)(x) \ge a_1$ and if $x \in X - \zeta(e_1)$, then $\eta(s)(x) = a_0$. Since $a_0 < a_1$, we see that $\eta(s) \ne 0$.

We next show that $\eta[S]$ is dense in $FN(X)$. Let $f,g \in FN(X)$ with $f \prec_X g$. Suppose $a_0 < \cdots < a_n$ contain all the values of $f$ and $g$. Set $U_i=f^{-1}(\up a_i)$ and $V_i=f^{-1}(\up a_i)$. From $f \prec_X g$ it follows that $U_i \prec V_i$ in $\mathcal{RO}(X)$. By \cite[Thm.~I.3.9]{deV62}, $\zeta[B]$ is dense in $\mathcal{RO}(X)$. Therefore, for each $i$ there is $e_i\in B$ with $U_i \prec \zeta(e_i) \prec V_i$, and as in the proof of Theorem~\ref{prox relation}(5), we may assume that the $e_i$ are decreasing. Set $s=a_0 + \sum_{i=1}^{n} (a_{i} - a_{i-1})e_{i}$ and $h=\eta(s)$. Then $h = a_0 + \sum_{i=1}^{n} (a_{i} - a_{i-1})\chi_{\zeta(e_{i})}$. Also, by Lemma~\ref{easy first}(1) and Remark~\ref{decreasing}, $f= a_0 + \sum_{i=1}^{n} (a_{i} - a_{i-1})\chi_{U_{i}}$ and $g = a_0 + \sum_{i=1}^{n} (a_{i} - a_{i-1}) \chi_{V_{i}}$. Since $U_i\prec\zeta(e_i)\prec V_i$ for each $i$, we see that $f \prec_X h \prec_X g$. Thus, $\eta[B]$ is dense in $FN(X)$.

It remains to show that $s \prec t$ in $S$ iff $\eta(s) \prec_X \eta(t)$. For this, we need the following claim.

\begin{claim}\label{claim}
Let $s \in S$ and set $f = \eta(s)$. For each $a \in A$, we have $f^{-1}({\uparrow}a) \in \zeta[B]$.
\end{claim}

\begin{proof}[Proof of Claim:]
Write $s = a_0 + \sum_{i=1}^{n} (a_i - a_{i-1}) e_{i}$ in decreasing form. Then $f = a_0 + \sum_{i=1}^{n} (a_i - a_{i-1}) \chi_{\zeta(e_{i})}$. Let $a \in A$. Then $f^{-1}({\uparrow}a)$ is either empty or equal to $f^{-1}({\uparrow}a_i)$ for some $i$. As $f^{-1}({\uparrow}a_i) = \zeta(e_i) \in \zeta[B]$, the claim is proved.
\end{proof}

Now, let $s,t\in S$ and set $f = \eta(s)$ and $g = \eta(t)$. Suppose $a_0 < \cdots < a_n$ contain all the values of $f$ and $g$. Set $U_i = f^{-1}(\up a_i)$ and $V_i = g^{-1}(\up a_i)$. By Claim~\ref{claim}, $U_i, V_i \in \zeta[B]$. Write $U_i = \zeta(e_i)$ and $V_i = \zeta(k_i)$ with $e_i, k_i \in B$.

First suppose that $s \prec t$. Then $[(s-a_i) \vee 0] \wedge (a_i - a_{i-1}) \prec [(t-a_i) \vee 0] \wedge (a_i - a_{i-1})$. We have $\eta\left([(s-a_i) \vee 0] \wedge (a_i - a_{i-1})\right) = [(f-a_i) \vee 0] \wedge (a_i - a_{i-1})$ and $\eta\left([(t-a_i) \vee 0] \wedge (a_i - a_{i-1}) \right) =[(g-a_i) \vee 0] \wedge (a_i - a_{i-1})$.
It is easy to see that
\[
[(f-a_i) \vee 0] \wedge (a_i - a_{i-1}) = (a_i - a_{i-1})\chi_{U_{i}}
\]
and
\[
[(g-a_i) \vee 0] \wedge (a_i - a_{i-1}) = (a_i - a_{i-1})\chi_{V_{i}}.
\]
Because $\eta$ is 1-1,
\[
[(s-a_i) \vee 0] \wedge (a_i - a_{i-1}) = (a_i - a_{i-1})e_{i}
\]
and
\[
[(t-a_i) \vee 0] \wedge (a_i - a_{i-1}) = (a_i - a_{i-1}) k_{i}.
\]
Since the first is proximal to the second, we get $e_{i} \prec k_{i}$, and as $\zeta$ preserves proximity \cite[Ch.~I.3]{deV62}, $U_{i} \prec V_{i}$. Because this is true for each $i$, we conclude that $f \prec_X g$.

Conversely, suppose that $f \prec_X g$. Then $U_i \prec V_i$ for each $i$. Since $\zeta$ reflects proximity \cite[Ch.~I.3]{deV62}, $e_i \prec k_i$ for each $i$. By Lemma~\ref{easy first}(1), Remark~\ref{decreasing}, and the injectivity of $\eta$, we may write $s = a_0 + \sum_{i=1}^{n} (a_i - a_{i-1})e_{i}$ and $t = a_0 + \sum_{i=1}^{n} (a_i - a_{i-1})k_{i}$. From this we conclude that $s \prec t$.
\end{proof}

\begin{corollary}\label{cor:5.4}
If $S$ is a proximity Specker $A$-algebra, then any two proximities on $S$ that restrict to the same proximity on $\Id(S)$ are equal.
\end{corollary}

\begin{proof}
Let $\prec$ and $\prec'$ be two proximities on $S$ that restrict to the same proximity on $\Id(S)$. Let $X$ be the space of ends of $(\Id(S),\prec)$. By Theorem~\ref{rep}, there is an $\ell$-algebra embedding $\eta:S\rightarrow FN(X)$ such that $s \prec t$ iff $\eta(s) \prec_X \eta(t)$ for all $s,t \in S$. Another application of Theorem~\ref{rep} to $(S,\prec')$ shows that $s \prec' t$ iff $\eta(s) \prec_X \eta(t)$ for all $s,t \in S$. Thus, for all $s,t \in S$, we have $s \prec t$ iff $s \prec' t$, and hence $\prec$ and $\prec'$ are equal.
\end{proof}

\begin{corollary}\label{cor:5.5}
If $S$ is a Specker $A$-algebra, then each proximity on $\Id(S)$ extends to a unique proximity on $S$. Consequently, there is a 1-1 correspondence between proximities on $S$ and $\Id(S)$.
\end{corollary}

\begin{proof}
Let $X$ be the space of ends of $(\Id(S),\prec)$. As we saw in the proof of Theorem~\ref{rep}, there is an $\ell$-algebra embedding $\eta : S \to FN(X)$. For $s,t\in S$, define $s\prec t$ iff $\eta(s)\prec_X\eta(t)$. By Theorem~\ref{FN algebra}, $\prec_X$ is a proximity on $FN(X)$.  Therefore, $\prec$ satisfies (P1) through (P8). For (P9), let $s,t \in S$ with $s \prec t$. Set $f=\eta(s)$ and $g=\eta(t)$. Then $f \prec_X g$. Since $\eta[S]$ is dense in $FN(X)$, there is $r\in S$ such that $f \prec_X \eta(r) \prec_X g$. This implies $s\prec r\prec t$, as required. For (P10), let $0 < s$. Write $s = \sum_{i=1}^n a_i e_i$ in orthogonal form. Since $s>0$, some $a_i>0$. As $e_i \ne 0$, there is $k \in \Id(S)$ with $0 \ne k \prec e_i$. Because $S$ is torsion-free, $0 < a_i k \in S$ and $a_i k \prec a_i e_i \le s$. Thus, $0 < a_i k \prec s$. Consequently, $\prec$ is a proximity on $S$, and it follows from Corollary~\ref{cor:5.4} that it is the unique proximity extending $\prec$ on $\Id(S)$.
\end{proof}

\begin{corollary}\label{FN algebra char}
Let $(S,\prec)$ be a proximity $f$-algebra over $A$. Then there is an $\ell$-algebra isomorphism between $S$ and $FN(X)$, for some compact Hausdorff space $X$, that preserves and reflects the proximity iff $(S,\prec)$ is a proximity Baer Specker $A$-algebra.
\end{corollary}

\begin{proof}
That $(FN(X),\prec_X)$ is a proximity Baer Specker $A$-algebra follows from Theorem~\ref{FN algebra}. Conversely, suppose that $(S,\prec)$ is a proximity Baer Specker $A$-algebra. By Theorem~\ref{rep}, there is an $\ell$-algebra embedding $\eta:S\to FN(X)$ such that $s \prec t$ iff $\eta(s) \prec_X \eta(t)$ for all $s,t\in S$. Since $S$ is a Baer Specker $A$-algebra, $\Id(S)$ is a complete Boolean algebra \cite[Thm.~4.3]{BMMO13a}. Therefore, $(\Id(S),\prec)$ is a de Vries algebra, hence $\Id(S)$ is isomorphic to $\mathcal{RO}(X)$ \cite[Ch.~I.4]{deV62}. Thus, each $f\in FN(X)$ can be written in decreasing form $f= a_0 + \sum_{i=1}^{n} (a_i - a_{i-1})\chi_{\zeta(e_{i})}$, and setting $s= a_0 + \sum_{i=1}^{n} (a_i - a_{i-1})e_{i}$ yields $s\in S$ such that $\eta(s)=f$. Consequently, $\eta$ is an $\ell$-algebra isomorphism such that $s \prec t$ iff $\eta(s) \prec_X \eta(t)$ for all $s,t\in S$.
\end{proof}

\section{Proximity morphisms and continuous maps}

In this section we show that a continuous map between compact Hausdorff spaces gives rise to what we term a proximity morphism between the corresponding proximity Specker $A$-algebras. We also characterize proximity morphisms between proximity Specker $A$-algebras $(S,\prec)$ and $(T,\prec)$ by means of de Vries morphisms from $(\Id(S),\prec)$ to $(\Id(T),\prec)$, and by means of continuous maps between the corresponding dual compact Hausdorff spaces.

Let $\varphi : X \to Y$ be a continuous map between compact Hausdorff spaces. By de Vries duality \cite{deV62}, $\widehat{\varphi} : \mathcal{RO}(Y) \to \mathcal{RO}(X)$, given by $\widehat{\varphi}(U) = {\sf Int}\left({\sf Cl}(\varphi^{-1}(U))\right)$, is a de Vries morphism.

Define $\varphi^+: F(Y) \to F(X)$ by $\varphi^+(f) = f\circ \varphi$.
\[
\xymatrix{
X \ar[r]^{\varphi} \ar[dr]_{f \circ \varphi} & Y \ar[d]^f \\
& A
}
\]
It is straightforward to see that $\varphi^+$ is an $A$-algebra homomorphism. Next define $\varphi^*: FN(Y) \to FN(X)$ by $\varphi^*(f) = (\varphi^+(f))^\#=(f\circ \varphi)^\#$. By the definition of normalization,
\[
\varphi^*(f)^{-1}(\up a) = ((f \circ \varphi)^\#)^{-1}({\uparrow}a) = {\sf Int}\left({\sf Cl}\left(\varphi^{-1}(f^{-1}({\uparrow}a))\right)\right) = \widehat{\varphi}(f^{-1}({\uparrow}a)).
\]

In order to see what properties $\varphi^*$ satisfies, we require two lemmas. The following is the proximity Specker analogue of a well-known fact about proximity Boolean algebras.

\begin{lemma}\label{lem:6.1}
Let $(S,\prec)$ be a proximity Specker $A$-algebra. Then each $s\in S$ is the least upper bound of $\{t\in S:t\prec s\}$.
\end{lemma}

\begin{proof}
We first show that if $E$ is a subset of $\Id(S)$ and $e \in \Id(S)$ is the join of $E$ in $\Id(S)$, then $e$ is the join of $E$ in $S$. For, since the partial order on $S$ restricts to the usual Boolean order on $\Id(S)$ (Lemma~\ref{lem:3.9}(2)), $e \in S$ is an upper bound of $E$ in $S$. Suppose $s \in S$ is another upper bound of $E$. As $s\wedge e$ is also an upper bound of $E$ in $S$, without loss of generality we may assume that $s \le e$. Then, for each $k \in E$, we have $k \le s \le e$. Consider the map $\eta : S \to FN(X)$ of Theorem~\ref{rep}. We have $\eta(k) \le \eta(s) \le \eta(e)$ and $\eta(k), \eta(e)$ are idempotents in $FN(X)$. By Lemma~\ref{idempotents of FN}, there exist regular open sets $U$ and $V$ such that $\eta(k) = \chi_U$ and $\eta(e) = \chi_V$. Therefore, $\eta(s)$ must be the characteristic function of some subset between $U$ and $V$. By normality and Remark~\ref{immediate facts}(2), it must be the characteristic function of a regular open set, and so $\eta(s)$ is an idempotent in $FN(X)$. As $\eta$ is an $\ell$-algebra embedding, $s$ is an idempotent in $S$. Thus, $s =e$, and so $e$ is the join of $E$ in $S$.

Next let $s \in S$ and write $s = a_0 + \sum_{i=1}^n b_i e_i$ in decreasing form with $b_1,\dots,b_n > 0$. It is clear that $s$ is an upper bound of $\{ t \in S : t \prec s \}$. If $E_i$ is the set of idempotents $k_i$ with $k_i \prec e_i$, then by the argument above, $e_i$ is the join of $E_i$ in $S$. By \cite[\S2, Thms.~2.3 and 2.6]{Nak50},
\[
s = a_0 + \sum_{i=1}^n b_i e_i = a_0 + \sum_{i=1}^n b_i\left(\bigvee E_i\right) = \bigvee\{ a_0 + \sum_{i=1}^n b_ik_i : k_i \in E_i\}.
\]
Since $a_0 + \sum_{i=1}^n b_ik_i \prec s$, it follows that $s=\bigvee\{ t \in S : t \prec s\}$.
\end{proof}

\begin{lemma}\label{lem:6.3}
Let $\varphi:X\to Y$ be continuous, $f\in FN(Y)$, and $a_0<\cdots<a_n$ be the values of $f$. Write $f=a_0+\sum_{i=1}^{n}(a_{i}-a_{i-1})\chi_{U_{i}}$ in decreasing form, where $U_i=f^{-1}(\up a_i)$ are regular open. Then $\varphi^*(f)=a_0+\sum_{i=1}^{n}(a_{i}-a_{i-1})\varphi^*(\chi_{U_{i}})$.
\end{lemma}

\begin{proof}
By Lemma~\ref{easy first}(1), $f(\varphi(x))=a_i$ iff $x\in\varphi^{-1}(U_i)-\varphi^{-1}(U_{i+1})$ for each $x\in X$. Therefore, in $F(X)$ we have $\varphi^+(f)=f\circ\varphi=a_0 +\sum_{i=1}^{n}(a_{i}-a_{i-1})\chi_{\varphi^{-1}(U_{i})}$. Thus, by Remark~\ref{immediate facts}(2), we have $\varphi^*(f)=(f\circ\varphi)^\# = a_0 +\sum_{i=1}^{n}(a_{i}-a_{i-1})\chi_{\widehat{\varphi}(U_{i})}$. Since $\varphi^*(\chi_{U_i})=\widehat{\varphi}(U_i)$, we conclude that $\varphi^*(f) = a_0 +\sum_{i=1}^{n}(a_{i}-a_{i-1})\varphi^*(\chi_{U_{i}})$.
\end{proof}

The next proposition will motivate the definition of a proximity morphism.

\begin{proposition} \label{induced map on FN}
Let $\varphi : X \to Y$ be a continuous map between compact Hausdorff spaces. Then $\varphi^*:FN(Y)\to FN(X)$ satisfies the following properties for each $f,g\in FN(Y)$ and $a\in A$.
\begin{enumerate}
\item $\varphi^*(0) = 0$.
\item $\varphi^*(f \wedge g) = \varphi^*(f) \wedge \varphi^*(g)$.
\item $f \prec_Y g$ implies $-\varphi^*(-f) \prec_X \varphi^*(g)$.
\item $\varphi^*(f)$ is the least upper bound of $\{\varphi^*(g) : g \prec_Y f\}$.
\item $\varphi^*(f + a) = \varphi^*(f) + a$.
\item If $a$ is positive, then $\varphi^*(af) = a\varphi^*(f)$.
\item $\varphi^*(f \vee a) = \varphi^*(f) \vee a$.
\end{enumerate}
\end{proposition}

\begin{proof}
(1) $\varphi^*(0) = (\varphi^+(0))^\# = 0^\# = 0$ since $0$ is continuous, hence normal. Note that the same argument shows $\varphi^*(a) = a$ for each $a \in A$.

(2) Let $f, g \in FN(Y)$. Recalling that meet in $FN(Y)$ is pointwise, we see that $\varphi^*(f \wedge g) = ((f \wedge g)\circ \varphi)^\# = ((f\circ \varphi) \wedge (g \circ \varphi))^\#$ and $\varphi^*(f) \wedge \varphi^*(g) = (f\circ \varphi)^\# \wedge (g \circ \varphi)^\#$. Therefore, it suffices to prove that $(h \wedge k)^\#=h^\# \wedge k^\#$ for each $h,k \in F(Y)$. Let $a \in A$. Then
\begin{eqnarray*}
(h^\# \wedge k^\#)^{-1}({\uparrow}a) &=& (h^\#)^{-1}({\uparrow}a) \cap (k^\#)^{-1}({\uparrow}a) \\
&=& {\sf Int}\left({\sf Cl}\left(h^{-1}({\uparrow}a)\right)\right) \cap {\sf Int}\left({\sf Cl}\left(k^{-1}({\uparrow}a)\right)\right) \\
&=& {\sf Int}\left({\sf Cl}\left(h^{-1}({\uparrow}a) \cap k^{-1}({\uparrow}a)\right)\right) \\
&=& {\sf Int}\left({\sf Cl}(h \wedge k)^{-1}({\uparrow}a)\right) \\
&=& ((h \wedge k)^\#)^{-1}\left({\uparrow}a\right).
\end{eqnarray*}
Thus, $\varphi^*(f \wedge g) = \varphi^*(f) \wedge \varphi^*(g)$.

(3) Let $f, g \in FN(Y)$ with $f \prec_Y g$. Suppose the values of $f$ and $g$ are among $a_0 < \cdots < a_n$. Then $f^{-1}(\up a_i) \prec g^{-1}(\up a_i)$ for each $i$. Therefore, since $\widehat{\varphi}$ is a de Vries morphism, by (M3), $\lnot\widehat{\varphi}\left(\lnot f^{-1}(\up a_i)\right)\prec\widehat{\varphi}\left(g^{-1}(\up a_i)\right)$ for each $i$. By Lemma~\ref{easy first}(1) and Remark~\ref{decreasing}, write $f=a_0 +\sum_{i=1}^{n}(a_i - a_{i-1})\chi_{U_{i}}$ and $g=a_0 + \sum_{i=1}^{n}(a_i - a_{i-1})\chi_{V_{i}}$ in decreasing form, where $U_i=f^{-1}(\up a_i)$ and $V_i=g^{-1}(\up a_i)$ are regular open and the sums are pointwise. Set $b_{i}=a_i - a_{i-1}$ and $b=\sum_{i=1}^{n} b_{i}$. We have
\begin{eqnarray*}
-f &=& -a_0 + \sum_{i=1}^{n} -b_{i}\chi_{U_{i}} = -a_0+b-b+\sum_{i=1}^{n} -b_{i}\chi_{U_{i}} \\
&=& -a_0 -b + \sum_{i=1}^{n} b_{i}(1-\chi_{U_{i}}) = -(a_0 + b)+\sum_{i=1}^{n} b_{i}\chi_{\lnot U_{i}}.
\end{eqnarray*}
This writes $-f$ in decreasing form since $\lnot U_n\supseteq\cdots\supseteq\lnot U_1$. Therefore, by Lemma~\ref{lem:6.3}, $\varphi^*(-f)=-(a_0 + b)+\sum_{i=1}^{n} b_{i}\varphi^*(\chi_{\lnot U_{i}})$. Thus,
\begin{eqnarray*}
-\varphi^*(-f) &=& (a_0+b)+\sum_{i=1}^{n} -b_{i}\varphi^*(\chi_{\lnot U_{i}}) \\
&=& (a_0+b)-b+b+\sum_{i=1}^{n} -b_{i}\varphi^*(\chi_{\lnot U_{i}}) \\
&=& a_0+\sum_{i=1}^{n} b_{i}(1-\varphi^*(\chi_{\lnot U_{i}})) \\
&=& a_0+\sum_{i=1}^{n} b_{i}(\lnot\varphi^*(\chi_{\lnot U_{i}})).
\end{eqnarray*}
Since $\lnot\widehat{\varphi}\left(\lnot U_{i}\right)\prec\widehat{\varphi}\left(V_{i}\right)$, we have $\lnot\varphi^*(\chi_{\lnot U_{i}})\prec_X\varphi^*(\chi_{V_{i}})$, yielding $-\varphi^*(-f)\prec_X\varphi^*(g)$, as required.

(4) Let $f \in FN(Y)$. We claim that $\varphi^*(f)^{-1}(\up a) = \bigvee \{ \varphi^*(g)^{-1}(\up a) : g \prec_Y f\}$. From the connection between $\varphi^*$ and $\widehat{\varphi}$, this amounts to proving $\widehat{\varphi}(f^{-1}(\up a)) = \bigvee \{ \widehat{\varphi}(g^{-1}(\up a)) : g \prec_Y f\}$. That $\widehat{\varphi}(f^{-1}(\up a))$ is an upper bound of $\{\widehat{\varphi}(g^{-1}(\up a)) : g \prec_Y f \}$ is clear. Conversely, first suppose that $0 \le f$. Then $0 \prec_Y f$. If $a \le 0$, then $\widehat{\varphi}(f^{-1}(\up a)) = X$, and as $0^{-1}(\up a) = X$, the claim is true in this case. Now suppose that $a \ge 0$. Let $U \in \mathcal{RO}(X)$ with $U \prec f^{-1}(\up a)$. Then $h := a\chi_U$ satisfies $h \prec_Y f$ and $h^{-1}(\up a) = U$. Because $\widehat{\varphi}(f)^{-1}(\up a) = \bigvee \{ \widehat{\varphi}(U) : U \prec f^{-1}(\up a) \}$, the claim holds for $f \ge 0$. For an arbitrary $f$, since $f$ is finitely valued, there is $b \in A$ with $0 \le f+b$. By Remark~\ref{special case}, $f+b$ is pointwise, so the case just done gives
\begin{eqnarray*}
f^{-1}(\up a) &=& (f+b)^{-1}(\up(a+b)) = \bigvee \{ h^{-1}(\up(a+b)) : h \prec_Y f+b\}  \\
&=& \bigvee \{ (h-b)^{-1}(\up a) : h \prec_Y f+b \} = \bigvee \{ g^{-1}(\up a) : g \prec_Y f\}.
\end{eqnarray*}
The last equality follows since $b \prec b$ for all $b \in A$ (Remark~\ref{consequences of proximity definition}).

For proving (5), (6), and (7) we use Remark~\ref{special case} which gives that addition and join by a scalar and multiplication by a positive scalar are pointwise.

(5) Let $f \in FN(Y)$ and $a \in A$. Then
\[
\varphi^*(f+a) = ((f+a)\circ \varphi)^\# = (f\circ \varphi + a)^\# = (f\circ \varphi)^\# + a = \varphi^*(f) + a.
\]

(6)  Let $f \in FN(Y)$ and $a \in A$ be positive. Then
\[
\varphi^*(af) = ((af)\circ \varphi)^\# = (a(f \circ \varphi))^\# = a(f \circ \varphi)^\# = a\varphi^*(f).
\]

(7) Let $f \in FN(Y)$ and $a \in A$. Then
\[
\varphi^*(f \vee a) = ((f \vee a) \circ \varphi)^\# = ((f \circ \varphi) \vee a)^\# = (f \circ \varphi)^\# \vee a = \varphi^*(f) \vee a.
\]
\end{proof}

Proposition~\ref{induced map on FN} motivates the following definition.

\begin{definition} \label{def:5.1}
Let $(S,\prec)$ and $(T,\prec)$ be proximity $f$-algebras over $A$. A map $\alpha : S \to T$ is a \emph{proximity morphism} provided for each $s,t\in S$ and $a\in A$, we have:
\begin{enumerate}
\item $\alpha(0) = 0$.
\item $\alpha(s \wedge t) = \alpha(s) \wedge \alpha(t)$.
\item $s \prec t$ implies $-\alpha(-s) \prec \alpha(t)$.
\item $\alpha(s)$ is the least upper bound of $\{\alpha(t) : t \prec s\}$.
\item $\alpha(s + a) = \alpha(s) + a$.
\item If $a$ is positive, then $\alpha(as) = a\alpha(s)$.
\item $\alpha(s \vee a) = \alpha(s) \vee a$.
\end{enumerate}
\end{definition}

\begin{remark} \label{notes on proximity morphisms}
\begin{enumerate}
\item[]
\item It follows from (1) and (5) that $\alpha(a) = a$ for each $a \in A$. In particular, $\alpha(1)=1$, so if $T$ is nontrivial, then $0\ne 1$ in $T$, and hence $\alpha$ is nonzero.
\item It follows from (2) that $\alpha$ is order preserving. Also, for $s \in S$ and $a \in A$, we have $\alpha(s \wedge a) = \alpha(s) \wedge \alpha(a) = \alpha(s) \wedge a$.
\end{enumerate}
\end{remark}

\begin{proposition}\label{prop:5.4}
Let $(S,\prec)$ and $(T,\prec)$ be proximity Specker $A$-algebras and let $\alpha : S \to T$ be a proximity morphism. Then $\alpha(\func{Id}(S)) \subseteq \func{Id}(T)$ and $\alpha|_{\func{Id}(S)}:\func{Id}(S)\to\func{Id}(T)$ is a de Vries morphism.
\end{proposition}

\begin{proof}
If $T$ is trivial, there is nothing to verify, so assume that $T$ is nontrivial. Let $X$ be the space of ends of $\Id(T)$, and let $\eta_T:T\to FN(X)$ be the $\ell$-algebra embedding of Theorem~\ref{rep}. Suppose $e \in \func{Id}(S)$. Because $\alpha$ is order preserving with $\alpha(0) = 0$ and $\alpha(1) = 1$, we see that $0 \le \alpha(e) \le 1$. Take $x \in X$ and set $a = \eta_T(\alpha(e))(x)$. Then $0 \le a \le 1$. By Lemma~\ref{lem:3.9}(3), $a\wedge e=ae$, so
\[
a \wedge \alpha(e) = \alpha(a \wedge e) = \alpha(ae) = a\alpha(e).
\]
Therefore, $a \wedge \eta_T(\alpha(e))=a\eta_T(\alpha(e))$. Evaluating at $x$ yields $a = a^2$. Thus, as $A$ is a domain, $a = \eta_T(\alpha(e))(x) \in \{0,1\}$. This shows that $\eta_T(\alpha(e)) \in \Id(FN(X))$. Since $\eta_T$ is an $A$-algebra homomorphism, this implies that $\alpha(e) \in \Id(T)$. It follows that $\alpha|_{\func{Id}(S)}: \func{Id}(S) \to \func{Id}(T)$ is well defined. It is also clear that $\alpha|_{\func{Id}(S)}$ satisfies (M1) and (M2). Suppose that $e,k\in\func{Id}(S)$ with $e \prec k$. Then $\lnot \alpha(\lnot e)=1-\alpha(1-e)=1-[1+\alpha(-e)]=-\alpha(-e)$. Because $-\alpha(-e) \prec \alpha(k)$, we conclude that $\lnot \alpha(\lnot e) \prec \alpha(k)$. Therefore, $\alpha|_{\func{Id}(S)}$ satisfies (M3). Let $k \in \func{Id}(S)$. Then $\alpha(k)$ is the least upper bound of $\{\alpha(s) : s \in S, s \prec k\}$. Suppose that $0 \le s \prec k$. Write $s = \sum_{i=1}^n a_i e_i$ in orthogonal form with each $a_i \ne 0$. The proof of Proposition~\ref{prop:4.6} then yields $0 < a_i\le 1$ and $e_i \prec k$ for each $i$. Consequently, $s \le e_1 \vee \cdots \vee e_n \prec k$. Since $\alpha(s) \le \alpha(e_1 \vee \cdots \vee e_n)$, we see that $\alpha(k)=\bigvee\{\alpha(e) : e \in \func{Id}(S), e \prec k\}$. Thus, $\alpha|_{\func{Id}(S)}$ satisfies (M4).
\end{proof}

The next theorem, which is the main result of this section, characterizes proximity morphisms.

\begin{theorem} \label{characterization of proximity morphism}
Suppose that $(S,\prec)$ and $(T,\prec)$ are proximity Specker $A$-algebras and $\alpha:S\to T$ is a map. Let $X$ be the space of ends of $(\Id(T), \prec)$ and $Y$ be the space of ends of $(\Id(S),\prec)$. Then the following conditions are equivalent.
\begin{enumerate}
\item $\alpha$ is a proximity morphism.
\item The restriction $\alpha|_{\Id(S)}:\Id(S)\to\Id(T)$ is a well-defined de Vries morphism, and if $s = a_0 + \sum_{i=1}^n b_i e_i$ is in decreasing form with $b_1,\dots,b_n > 0$, then $\alpha(s) = a_0 + \sum_{i=1}^n b_i \alpha(e_i)$.
\item There exists a continuous map $\varphi:X\to Y$ such that the following diagram commutes.
\[
\xymatrix{
S \ar[r]^{\alpha} \ar[d]_{\eta_S} & T \ar[d]^{\eta_T} \\
FN(Y) \ar[r]_{\varphi^*} & FN(X)
}
\]
\end{enumerate}
\end{theorem}

\begin{proof}
(1)$\Rightarrow$(2): By Proposition~\ref{prop:5.4}, $\alpha|_{\Id(S)}$ is a well-defined de Vries morphism. Let $a_i = a_0 + \sum_{j=1}^i b_j$. Then $a_0 < \cdots < a_n$ and
\begin{eqnarray*}
s - a_{i-1} &=& (a_0 + b_1e_1 + \cdots + b_ne_n) - a_{i-1} \\
&=& (b_ie_i + \cdots + b_ne_n) - (a_{i-1} - a_0 - b_1e_1 - \cdots - b_{i-1}e_{i-1}) \\
&=& (b_ie_i + \cdots + b_ne_n) - (a_0 + b_1 + \cdots + b_{i-1} - a_0 - b_1e_1 - \cdots - b_{i-1}e_{i-1}) \\
&=& (b_ie_i + \cdots + b_ne_n) - (b_1(1-e_1) + \cdots + b_{i-1}(1-e_{i-1})) \\
&=& (b_ie_i + \cdots + b_ne_n) - (b_1\lnot e_1 + \cdots + b_{i-1}\lnot e_{i-1}).
\end{eqnarray*}
This exhibits $s-a_{i-1}$ as the difference of two elements greater than or equal to 0, and we claim their meet is 0. To see this, as the $e_i$ are decreasing, we have
\begin{eqnarray*}
0 & \le & (b_ie_i + \cdots + b_ne_n) \wedge (b_1\lnot e_1 + \cdots + b_{i-1}\lnot e_{i-1}) \\
& \le & (b_ie_i + \cdots + b_ne_i) \wedge (b_1\lnot e_{i-1} + \cdots + b_{i-1}\lnot e_{i-1}) \\
&=& (b_i + \cdots + b_n)e_i \wedge (b_1 + \cdots + b_{i-1})\lnot e_{i-1},
\end{eqnarray*}
which is 0 by using the $f$-ring identity twice along with $e_i\wedge\lnot e_{i-1}=0$. Therefore, by \cite[Ch.~XIII, Lem.~4]{Bir79}, $(s-{a_{i-1}}) \vee 0 = b_i e_i + \cdots + b_n e_n$. As was shown in the proof of Theorem~\ref{rep}, $[(s-a_{i-1})\vee 0]\wedge b_i = b_i e_i$.

\begin{claim}\label{claim_2}
Let $s \in S$ and let $a,b \in A$ with $a < b$. Then $(s \wedge b) - (s \wedge a) =  ((s-a) \vee 0) \wedge(b-a)$.
\end{claim}

\begin{proof}
We have
\begin{eqnarray*}
(s\wedge b) - (s \wedge a) &=& ((s \wedge b) - a) - ((s \wedge a) - a) \\
&=& ((s-a) \wedge (b-a)) - ((s-a) \wedge (a-a)) \\
&=& ((s-a) \wedge (b-a)) - ((s-a) \wedge 0).
\end{eqnarray*}
Write $t = (s-a) \wedge (b-a)$. Then $(s\wedge b) - (s\wedge a) = t - (t\wedge 0)$ because $0 \le b-a$. By  \cite[Ch.~XIII, Thm.~7]{Bir79}, $t = (t\vee 0) + (t \wedge 0)$. Therefore, $t- (t \wedge 0)  = t\vee 0$. Thus, $(s\wedge b) - (s\wedge a) = t \vee 0 = ((s-a) \wedge (b-a)) \vee 0 = ((s-a) \vee 0) \wedge (b-a)$ since $S$ is a distributive lattice (\cite[Ch.~ XIII, Thm.~4]{Bir79}).
\end{proof}

As $\alpha([(s-a_{i-1})\vee 0]\wedge b_i) = [(\alpha(s)-a_{i-1})\vee 0]\wedge b_i$ and $\alpha(b_ie_i)=b_i\alpha(e_i)$, we obtain $b_i\alpha(e_i) = [(\alpha(s)-a_{i-1})\vee 0]\wedge b_i$. Since $b_i=a_i-a_{i-1}$, by Claim~\ref{claim_2}, $[(\alpha(s)-a_{i-1})\vee 0]\wedge b_i = (\alpha(s)\wedge a_i) - (\alpha(s) \wedge a_{i-1})$. As $a_0 \le s \le a_n$, we have $a_0 \le \alpha(s) \le a_n$. Consequently,
\[
\alpha(s) - a_0 = (\alpha(s)\wedge a_n) - (\alpha(s)\wedge a_0) = \sum_{i=1}^n ((\alpha(s)\wedge a_i) - (\alpha(s)\wedge a_{i-1})) = \sum_{i=1}^n b_i \alpha(e_i).
\]
Adding $a_0$ to both sides of the equation finishes the proof.

(2)$\Rightarrow$(3): Let $\varphi : X \to Y$ be the dual of the de Vries morphism $\alpha|_{\Id(S)}$. First let $e \in \Id(S)$. By \cite[Ch.~I.6]{deV62}, the following diagram commutes.
\[
\xymatrix{
\Id(S) \ar[r]^{\alpha|_{\Id(S)}} \ar[d]_{\zeta_S} & \Id(T) \ar[d]^{\zeta_T} \\
\mathcal{RO}(Y) \ar[r]_{\widehat{\varphi}} & \mathcal{RO}(X)
}
\]
We have $\varphi^*(\chi_U) = \chi_{\widehat{\varphi}(U)}$. This implies $\varphi^*(\eta_S(e)) = \eta_T(\alpha(e))$ for each $e \in \Id(S)$. Now, let $s \in S$ be arbitrary. Write $s = a_0 + \sum_{i=1}^n b_i e_i$ in decreasing form. Then, since $\eta_S,\eta_T$ are $A$-algebra homomorphisms, we have
\begin{eqnarray*}
\eta_T(\alpha(s)) &=& \eta_T(a_0 + \sum_{i=1}^n b_i \alpha(e_i)) = a_0 + \sum_{i=1}^n b_i \eta_T(\alpha(e_i)) \\
&=& a_0 + \sum_{i=1}^n b_i \varphi^*(\eta_S(e_i)) = \varphi^*(\eta_S(s)).
\end{eqnarray*}
Here the first equality follows from (2) and the last equality from Lemma~\ref{lem:6.3}. This yields commutativity of the diagram in the statement of (3).

(3)$\Rightarrow$(1): By Proposition~\ref{induced map on FN}, $\varphi^*$ is a proximity morphism. Because $\eta_S$ and $\eta_T$ are $\ell$-algebra embeddings which preserve and reflect proximity, all the proximity morphism axioms but the fourth are clearly true for $\alpha$. To prove axiom (4), let $s \in S$. Since $\varphi^*$ is a proximity morphism, $\varphi^*(\eta_S(s)) = \bigvee \{ \varphi^*(f) : f \prec_Y \eta_S(s) \}$. Because $\eta_S[S]$ is dense in $FN(Y)$ and $\eta_S$ reflects proximity, we have
\[
\varphi^*(\eta_S(s)) = \bigvee \{ \varphi^*(\eta_S(u)) : \eta_S(u) \prec_Y \eta_S(s) \} = \bigvee \{ \varphi^*(\eta_S(u)) : u \prec s \}.
\]
By (3) this yields $\eta_T(\alpha(s)) = \bigvee \{ \eta_T(\alpha(u)) : u \prec s \}$. Now, since $\eta_T$ is order reflecting, $\alpha(s)$ is an upper bound of $\{ \alpha(u) : u \prec s \}$. To see it is the least upper bound, let $t \in T$ satisfy $\alpha(u) \le t$ for each $u \prec s$. Then $\eta_T(\alpha(u)) \le \eta_T(t)$. Therefore, $\eta_T(\alpha(s)) \le \eta_T(t)$. Thus, $\alpha(s) \le t$. Consequently, $\alpha(s) = \bigvee \{ \alpha(u) : u \prec s\}$. This finishes the proof that $\alpha$ is a proximity morphism.
\end{proof}

\begin{corollary} \label{unique extension of map}
Let $(S,\prec)$ and $(T,\prec)$ be proximity Specker $A$-algebras. If $\sigma : \Id(S) \to \Id(T)$ is a de Vries morphism, then there is a unique proximity morphism $\alpha : S \to T$ with $\alpha|_{\Id(S)} = \sigma$.
\end{corollary}

\begin{proof}
As we pointed out in the beginning of Section 5, each $s \in S$ has a unique decreasing decomposition $s=a_0+\sum_{i=1}^n b_ie_i$, where $b_1,\dots,b_n>0$ and $1>e_1>\cdots>e_n>0$. Define $\alpha:S\to T$ by $\alpha(s)=a_0+\sum_{i=1}^n b_i\sigma(e_i)$. Since the decreasing decomposition is unique, $\alpha$ is well defined, and it follows from the definition that $\alpha|_{\Id(S)} = \sigma$. If $\alpha^\prime$ is another proximity morphism extending $\sigma$, then Theorem~\ref{characterization of proximity morphism} implies that $\alpha^\prime(s) = a_0 + \sum_{i=1}^n b_i \sigma(e_i) = \alpha(s)$, and thus $\alpha$ is the unique proximity morphism extending $\sigma$.
\end{proof}

\begin{corollary} \label{1-1 correspondence}
With the notation of Theorem~\ref{characterization of proximity morphism} and with $S$ and $T$ Baer, there are 1-1 correspondences between the proximity morphisms $S \to T$, the de Vries morphisms $\Id(S) \to \Id(T)$, and the continuous maps $X \to Y$.
\end{corollary}

\begin{proof}
By Corollary~\ref{unique extension of map}, $\alpha \mapsto \alpha|_{\Id(S)}$ is a 1-1 correspondence between the proximity morphisms $S \to T$ and the de Vries morphisms $\Id(S) \to \Id(T)$. By \cite[Ch.~I.6]{deV62}, there is a 1-1 correspondence between the de Vries morphisms $\Id(S) \to \Id(T)$ and the continuous maps $X \to Y$.
\end{proof}

\section{The space of ends of a proximity Specker algebra}

As we have seen, for a proximity Specker $A$-algebra $(S,\prec)$, the space of ends of $(\Id(S),\prec)$ is useful for representing $S$ as a ring of normal functions. In this section we pursue this further by developing the notion of ends for a proximity Specker $A$-algebra $(S,\prec)$.

For a continuous map $\varphi : X \to Y$ between compact Hausdorff spaces, we recall from the previous section that the de Vries morphism $\widehat{\varphi} : \mathcal{RO}(Y) \to \mathcal{RO}(X)$, given by $\widehat{\varphi}(U) = {\sf Int}({\sf Cl}(\varphi^{-1}(U)))$, and the proximity morphism $\varphi^* : FN(Y) \to FN(X)$, given by $\varphi^*(f) = (f \circ \varphi)^\#$, are connected by the formula
\[
\varphi^*(f)^{-1}( {\uparrow}a) = \widehat{\varphi}(f^{-1}( {\uparrow}a)).
\]
We also recall that since $\widehat{\varphi}$ is a de Vries morphism, it has the property that whenever $e_i \prec k_i$ for $1 \le i \le n$, then $\widehat{\varphi}(e_1 \vee \cdots \vee e_n) \prec \widehat{\varphi}(k_1) \vee \cdots \vee \widehat{\varphi}(k_n)$ \cite[Lem.~2.2]{Bez12}.

\begin{proposition}\label{prop:7.1}
\begin{enumerate}
\item[]
\item Let $\varphi:X\to Y$ be a continuous map between compact Hausdorff spaces and let $f,g,h,k \in FN(Y)$. If $f \prec_Y h$ and $g \prec_Y k$, then $\varphi^*(f+g) \prec_X \varphi^*(h) + \varphi^*(k)$.
\item Let $\alpha : S \to T$ be a proximity morphism between proximity Specker $A$-algebras and let $s,t,u,v \in S$. If $s \prec u$ and $t \prec v$, then $\alpha(s+t) \prec \alpha(u) + \alpha(v)$.
\end{enumerate}
\end{proposition}

\begin{proof}
(1) It is sufficient to prove that $\varphi^*(f+g)^{-1}({\uparrow}a) \prec (\varphi^*(h) + \varphi^*(k))^{-1}({\uparrow}a)$ for each $a \in A$. By Lemma~\ref{AlgebraicPropertiesOfFN}(1),
\[
\varphi^*(f+g)^{-1}({\uparrow}a) = \widehat{\varphi}\left((f+g)^{-1}(\up a)\right) =  \widehat{\varphi} \left( \bigvee_{b+c \ge a} f^{-1}({\uparrow}b) \cap g^{-1}({\uparrow}c) \right).
\]
On the other hand, from the same lemma, we have
\[
(\varphi^*(h) + \varphi^*(k))^{-1}({\uparrow}a) = \bigvee_{b+c \ge a} \varphi^*(h)^{-1}({\uparrow}b) \cap \varphi^*(k)^{-1}({\uparrow}c) =
\bigvee_{b+c \ge a} \widehat{\varphi}(h^{-1})({\uparrow}b) \cap \widehat{\varphi}(k^{-1})({\uparrow}c).
\]
Because $f \prec_Y h$ and $g \prec_Y k$, for each $b,c \in A$,  we have $f^{-1}({\uparrow}b) \prec h^{-1}({\uparrow}b)$ and $g^{-1}({\uparrow}c) \prec k^{-1}({\uparrow}c)$. Therefore, $f^{-1}({\uparrow}b) \cap g^{-1}({\uparrow}c) \prec h^{-1}({\uparrow}b) \cap k^{-1}({\uparrow}c)$. Thus, since $\widehat{\varphi}$ is a de Vries morphism and the joins in question are finite joins,
\[
\widehat{\varphi} \left( \bigvee_{b+c \ge a} f^{-1}({\uparrow}b) \cap g^{-1}({\uparrow}c)) \right) \prec \bigvee_{b+c \ge a} \widehat{\varphi}(h^{-1}({\uparrow}b) \cap k^{-1}({\uparrow}c)) = \bigvee_{b+c \ge a} \widehat{\varphi}(h^{-1})({\uparrow}b) \cap \widehat{\varphi}(k^{-1}({\uparrow}c)).
\]
Consequently, $\varphi^*(f+g) \prec_X \varphi^*(h) + \varphi^*(k)$.

(2) Consider the maps $\eta_S : S \to FN(Y)$ and $\eta_T : T \to FN(X)$ provided by Theorem~\ref{rep}. By Theorem~\ref{characterization of proximity morphism}, $\varphi^* \circ \eta_S = \eta_T \circ \alpha$. Since $\eta_S$ preserves proximity, $\eta_S(s) \prec_Y \eta_S(u)$ and $\eta_S(t) \prec_Y \eta_S(v)$. As $\eta_S, \eta_T$ preserve addition, (1) yields
\begin{eqnarray*}
\eta_T(\alpha(s+t)) &=& \varphi^*(\eta_S(s+t)) = \varphi^*(\eta_S(s) + \eta_S(t)) \prec_X \varphi^*(\eta_S(u)) + \varphi^*(\eta_S(v)) \\
&=& \eta_T(\alpha(u)) + \eta_T(\alpha(v)) = \eta_T(\alpha(u) + \alpha(v)).
\end{eqnarray*}
Thus, since $\eta_T$ reflects proximity, $\alpha(s+t) \prec \alpha(u) + \alpha(v)$.
\end{proof}

We recall that if $S$ is an $\ell$-ring, then the \emph{absolute value} of $s\in S$ is defined as $|s|=s\vee(-s)$. For each $s,t \in S$, we have $|s+t| \le |s| + |t|$ and $|st| \le |s| \cdot |t|$; moreover, if $S$ is an $f$-ring, then $|st|=|s|\cdot|t|$ (see, e.g., \cite[Ch.~XVII]{Bir79}). We also recall that an ideal $I$ of $S$ is an \emph{$\ell$-ideal} provided $|s|\le|t|$ and $t\in I$ imply $s\in I$ for all $s,t\in S$. An $\ell$-ideal $I$ is {\em proper} if $I\ne S$.

Let $(S,\prec)$ be a proximity Specker $A$-algebra. For $A\subseteq S$, set
$$
{\thd}A = \{ s \in S : |s| \prec a \text{ for some }a \in A\}.
$$

\begin{definition}
We call an $\ell$-ideal $I$ of a proximity Specker $A$-algebra $(S,\prec)$ a {\it round ideal} provided $I=S$ or ${\thd}I=I$ and $I \cap A = 0$. We call $I$ an {\it end} provided $I$ is maximal among proper round ideals of $(S,\prec)$.
\end{definition}

\begin{lemma}\label{kernel is round}
Let $\alpha : S \to T$ be a proximity morphism between proximity Specker $A$-algebras with $T$ nontrivial. If $I$ is an $\ell$-ideal of $T$ with $I \cap A = 0$, then $\thd \alpha^{-1}(I)$ is a proper round ideal of $S$.
\end{lemma}

\begin{proof}
Set $J = \thd \alpha^{-1}(I)$, and let $s,t \in J$. Then there are $u,v \in S$ with $|s| \prec u$, $|t| \prec v$, and $\alpha(u), \alpha(v) \in I$. Therefore, there are $u', v'$ with $|s| \prec u' \prec u$ and $|t| \prec v' \prec v$. We have $|s \pm t| \le |s| + |t| \prec u' + v'$. Since $\alpha$ is order preserving, $0 \le \alpha(u' + v')$. By Proposition~\ref{prop:7.1}(2), $\alpha(u' + v') \prec \alpha(u) + \alpha(v) \in I$. As $I$ is an $\ell$-ideal, $\alpha(u' + v') \in I$. This yields $u' + v' \in \alpha^{-1}(I)$, so $s\pm t \in J$. Next, let $s \in J$ and $t \in S$. Then $|s| \prec u$ for some $u \in S$ with $\alpha(u) \in I$. Since $S$ is a Specker $A$-algebra, write $t = \sum_{i=1}^n a_i e_i$ in orthogonal form, and observe that
\[
|t| = \left|\sum_{i=1}^n a_i e_i\right| \le \sum_{i=1}^n |a_i e_i| = \sum_{i=1}^n |a_i||e_i| \le \sum_{i=1}^n |a_i|,
\]
where the last inequality follows from Lemma~\ref{lem:3.9}(1). Therefore, there is $a \in A$ with $|t| \le a$. Then $|st| = |s| \cdot |t| \le a|s| \prec au$. As $\alpha(au) = a\alpha(u) \in I$, we see that $st \in J$. Thus, $J$ is an ideal of $S$. To see it is an $\ell$-ideal, let $t \in J$ and $s \in S$ satisfy $|s| \le |t|$. Then there is $u \in S$ with $|t| \prec u$ and $\alpha(u) \in I$. Therefore, $|s| \prec u$, and so $s \in J$.  We have thus proved that $J$ is an $\ell$-ideal of $S$. Next, take $s \in J$. Then $|s| \prec u$ for some $u\in S$ with $\alpha(u) \in I$. There is $t \in S$ with $|s| \prec t \prec u$. This implies $t\in J$, so $s \in \thd J$. Thus, $\thd J = J$. To see that $J \cap A = 0$, if $a \in J \cap A$, then $|a| \prec u$ for some $u \in S$ with $\alpha(u) \in I$. Therefore, $0 \le |a| = \alpha(|a|) \le \alpha(u) \in I$, so $|a| \in I$ since $I$ is an $\ell$-ideal. Because $I \cap A = 0$, we get $a = 0$. Consequently, $J \cap A = 0$, and so $J$ is a proper round ideal of $S$.
\end{proof}

For a proximity morphism $\alpha : S \to T$ between proximity Specker $A$-algebras, define the \emph{kernel} of $\alpha$ as
$$
\ker(\alpha) = {\thd}\alpha^{-1}(0).
$$
As noted in Remark~\ref{notes on proximity morphisms}(1), if $T$ is nontrivial, then $\alpha$ is nonzero. If $\alpha = 0$, it is clear that $\ker(\alpha)=S$. On the other hand, if $\alpha$ is nonzero, then $\alpha(a)=a$ for each $a\in A$.

\begin{proposition}\label{double down}
\begin{enumerate}
\item[]
\item Let $\alpha : S \to T$ be a proximity morphism with $T$ nontrivial. Then $\ker(\alpha)$ is a proper round ideal of $S$.
\item If $P$ is a minimal prime ideal of $S$, then ${\thd} P$ is a proper round ideal of $S$.
\end{enumerate}
\end{proposition}

\begin{proof}
The first statement follows from Lemma~\ref{kernel is round} since $(0)$ is an $\ell$-ideal of $T$. For the second statement, it is sufficient to observe that if $P$ is a minimal prime ideal of $S$, then $P$ is an $\ell$-ideal by \cite[p.~196]{Sub67}, and $P \cap A = 0$ by \cite[Lem.~4.5]{BMMO13a}.
\end{proof}

In Theorem~\ref{char of ends} we give several characterizations of ends of a proximity Specker $A$-algebra $(S,\prec)$. For this we require the following lemma. We recall that by Proposition~\ref{prop:4.6}, for a proximity Specker $A$-algebra $(S,\prec)$, the restriction of $\prec$ is a proximity on $\Id(S)$.

\begin{lemma}\label{compatible}
Let $(S,\prec)$ be a proximity Specker $A$-algebra. If $s, t \in S$, we may write $s = a_0 + \sum_{i=1}^n b_i e_i$ and $t = a_0 + \sum_{i=1}^n b_i k_i$ in compatible decreasing form with each $b_i > 0$. Moreover, $s \le t$ iff $e_i \le k_i$ for each $i$, and $s \prec t$ iff $e_i \prec k_i$ for each $i$. Furthermore, $0 \le s$ iff $0 \le a_0$.
\end{lemma}

\begin{proof}
Let $\eta : S \to FN(X)$ be the embedding of Theorem~\ref{rep}. By Lemmas~\ref{easy first}~and~\ref{refine} and Remark~\ref{decreasing}, we may write $\eta(s) = a_0 + \sum_{i=1}^n b_i \chi_{U_i}$ and $\eta(t) = a_0 + \sum_{i=1}^n b_i \chi_{V_i}$, where $X=U_0\supseteq U_1\supseteq\cdots\supseteq U_n\supseteq\varnothing$ and $X=V_0\supseteq V_1\supseteq\cdots\supseteq V_n\supseteq\varnothing$. Since $\eta(s),\eta(t)\in FN(X)$, the $U_i$ and $V_i$ are regular open subsets of $X$. Claim~\ref{claim} shows that there are $e_i, k_i \in \Id(S)$ with $\zeta(e_i) = U_i$ and $\zeta(k_i) = V_i$. Consequently, since $\eta$ is an $\ell$-algebra embedding, $s = a_0 + \sum_{i=1}^n b_ie_i$ and $t = a_0 + \sum_{i=1}^n b_i k_i$. Moreover, $s \le t$ iff $\eta(s) \le \eta(t)$, which by Lemma~\ref{refine} is equivalent to $U_i \subseteq V_i$ for each $i$. Since $\zeta:\Id(S)\to\mathcal{RO}(X)$ is an embedding, the last condition is equivalent to $e_i \le k_i$ for each $i$. Moreover, $s \prec t$ iff $\eta(s) \prec_X \eta(t)$, which is equivalent to $U_i \prec V_i$ for each $i$. As $\zeta$ preserves and reflects proximity, $U_i \prec V_i$ iff $e_i \prec k_i$ for each $i$. Finally, as each $b_i > 0$, the formula $s = a_0 + \sum_{i=1}^n b_i e_i$ implies that $0 \le s$ if $0\le  a_0$. For the converse, if $0 \le s$, then $0 \le \eta(s)$. By Lemma~\ref{easy first}(1) and Remark~\ref{decreasing}, $a_0$ is the smallest value of the function $\eta(s)$. Thus, $0 \le \eta(s)$ implies $0 \le a_0$. This finishes the proof.
\end{proof}

\begin{theorem}\label{char of ends}
Let $(S,\prec)$ be a nontrivial proximity Specker $A$-algebra and let $B=\Id(S)$. For an $\ell$-ideal $I$ of $S$, the following conditions are equivalent.
\begin{enumerate}
\item $I$ is an end of $(S,\prec)$.
\item $I$ is the kernel of a proximity morphism $\alpha : S \to A$.
\item $I = {\thd} P$ for some minimal prime ideal $P$ of $S$.
\item $I \cap B$ is an end of $(B,\prec)$ and $I$ is the ideal of $S$ generated by $I \cap B$.
\item $I$ is the ideal of $S$ generated by some end of $(B,\prec)$.
\end{enumerate}
\end{theorem}

\begin{proof}
(1)$\Rightarrow$(3): Let $I$ be an end of $(S,\prec)$. Then $I \cap A = 0$. A Zorn's lemma argument shows there is a prime ideal $P$ of $S$ with $I \subseteq P$ and $P \cap A = 0$. By \cite[Lem.~4.5]{BMMO13a}, $P$ is a minimal prime ideal of $S$. Because $I \subseteq P$ and $I$ is round, $I \subseteq {\thd} P$. Maximality of $I$ then forces $I = {\thd} P$ since $\thd P$ is a proper round ideal by Proposition~\ref{double down}(2).

(2)$\Rightarrow$(3): Let $I = \ker(\alpha)$ and let $\sigma$ be the restriction of $\alpha$ to $B$. Note that $\Id(A)$ is the two-element Boolean algebra {\bf 2}. By Proposition~\ref{prop:5.4}, $\sigma : B \to {\bf 2}$ is a de Vries morphism. Define $\ker(\sigma)={\thd}\sigma^{-1}(0)$. Obviously $\ker(\sigma) \subseteq \ker(\alpha) \cap B$. For the reverse inclusion, if $e \in \ker(\alpha) \cap B$, then there is $s \in S$ with $e \prec s$ and $\alpha(s) = 0$. By an argument  similar to the one given in the proof of Proposition~\ref{prop:4.6}, we can find $k \in B$ with $e \prec k$ and $ \sigma(k) = 0$. Therefore, $e \in \ker(\sigma)$, and so $\ker(\alpha) \cap B = \ker(\sigma)$. Thus, $E := I \cap B$ is an end of $(B,\prec)$ (see, e.g., \cite[Rem.~4.21]{BH13}). This implies that there is a maximal ideal $M$ of $B$ such that $E = \thd M$ (see, e.g., \cite[Sec.~3]{Bez10}). Let $P$ be the ideal of $S$ generated by $M$. By \cite[Lem.~2.5 and Thm.~2.7]{BMMO13a}, the Boolean homomorphism $B \to {\bf 2}$ whose kernel is $M$ extends to an $A$-algebra homomorphism $\beta : S \to A$, and its kernel contains $P$ since it contains $M$. Let $\beta(s)=0$ and write $s = \sum_{i=1}^n a_i e_i$ in orthogonal form with the $a_i\in A$ distinct and nonzero. Then $se_i = a_ie_i$, so $0=\beta(se_i)=\beta(a_ie_i)=a_i\beta(e_i)$. Since $S$ is torsion-free, this implies $\beta(e_i) = 0$. Therefore, $e_i \in M$, showing that $s \in P$. Thus, $P$ is the kernel of the onto ring homomorphism $\beta:S\to A$, and as $A$ is an integral domain, $P$ is a prime ideal. Since $P \cap A = 0$, by \cite[Lem.~4.5]{BMMO13a} $P$ is a minimal prime.

We wish to show $I = \thd P$. Take $s \in I$, and first suppose $s \ge 0$. There is $t \in S$ with $s \prec t$ and $\alpha(t) = 0$. By Lemma~\ref{compatible}, write $s = a_0 + \sum_{i=1}^n b_i e_i$ and $t = a_0 + \sum_{i=1}^n b_i k_i$ in compatible decreasing form with $b_i > 0$ and $e_i \prec k_i$ for each $i$. By Theorem~\ref{characterization of proximity morphism}, $0 = \alpha(t) = a_0 + \sum_{i=1}^n b_i \sigma(k_i)$. Since $\sigma(k_i) \in {\bf 2}$ and the $b_i$ are positive, this forces $a_0 = 0$ and all $\sigma(k_i) = 0$. Pick $l_i \in B$ with $e_i \prec l_i \prec k_i$. By replacing $l_i$ by $l_1 \wedge \cdots \wedge l_i$ we may assume the $l_i$ form a decreasing set. Let $r = \sum_{i=1}^n b_i l_i$. Then $s\prec r\prec t$ and $l_i \in \ker(\sigma) \subseteq M$. Consequently, $r \in P$, and so $s \in \thd P$. For an arbitrary $s$, the argument we just gave shows $|s| \in \thd P$. Since $P$ is a minimal prime, by Proposition~\ref{double down}(2), $\thd P$ is a round ideal, so $\thd P$ is an $\ell$-ideal, and hence $s \in \thd P$. This implies $I \subseteq \thd P$. For the reverse inclusion, let $s \in \thd P$. Then $|s| \prec t$ for some $t \in P$. To show $s \in I$, it suffices to show $|s| \in I$ since by Proposition~\ref{double down}(1), $I$ is an $\ell$-ideal. Therefore, assume $s \ge 0$. There is $r$ with $s \prec r \prec t$. Write $s = a_0 + \sum_{i=1}^n b_ie_i$, $r = a_0 + \sum_{i=1}^n b_i l_i$, and $t = a_0 + \sum_{i=1}^n b_i k_i$ in compatible decreasing form with $e_i \prec l_i \prec k_i$ for each $i$. Since $0 \le s$, Lemma~\ref{compatible} implies $0 \le a_0$. We show that each $k_i \in M$. As $t\in P$ and $P$ is generated by $M$, we can write $t = \sum_{j=1}^m c_j p_j$ for some $p_j \in M$. Set $c = \sum_{j=1}^m |c_j|$ and $p = p_1 \vee \cdots \vee p_m$. Then $cp \ge t \ge a_0 = a_0 \cdot 1 \ge 0$. By Lemma~\ref{lem:3.9}(6), this forces $a_0 = 0$ or $p = 1$. Since $p \in M$ and $1 \notin M$, we see that $a_0 = 0$. Furthermore, $cp \ge t \ge b_1k_1$. Applying Lemma~\ref{lem:3.9}(6) again yields $p \ge k_1 \ge k_i$ for each $i$. Since $p_j \in M$ for each $j$, we have $p \in M$, so each $k_i \in M$. Now, $l_i \prec k_i$ and $k_i \in M$ yield $l_i \in \thd M = E$. Therefore, $\alpha(r) = 0$, so as $s \prec r$, we obtain $s \in I$. Thus, $I = \thd P$.

(3)$\Rightarrow$(4): Let $I = {\thd} P$ for some minimal prime ideal $P$ of $S$, and set $E = I \cap B$. Then $E$ is an ideal of $B$. Let $M = P \cap B$. By \cite[Prop.~3.11 and Thm.~4.6]{BMMO13a}, $M$ is a maximal ideal of $B$. We show $E = {\thd} M$, which will prove that $E$ is an end of $(B,\prec)$. Let $e \in E$. Then $e \in I$, so there is $s \in P$ with $e \prec s$. An argument similar to the one given in the proof of Proposition~\ref{prop:4.6} gives $k \in M$ with $e \prec k$. This implies $e \in {\thd} M$, which yields $E \subseteq {\thd} M$. For the converse, let $e \in {\thd} M$. Then there is $k \in M$ with $e \prec k$. Therefore, $k \in P$, so $e \in {\thd} P = I$, and hence $e \in I\cap B = E$. This proves $E = {\thd} M$, so $E$ is an end of $(B,\prec)$. We next prove that $E$ generates $I$ as an ideal of $S$. One inclusion is obvious. For the reverse, let $s \in I$. Then there is $t \in P$ with $|s| \prec t$. By \cite[Ch.~XIII, Thm.~7]{Bir79}, $0\le s\vee 0\le |s|\prec t$, so $0\le s\vee 0\prec t$. By writing $s\vee 0$ and $t$ in compatible decreasing form, an argument similar to the one given in the proof of (1)$\Rightarrow$(3) yields that $s \vee 0$ is in the ideal generated by $E$. Since $-s\in I$, we have that $(-s) \vee 0 = -(s \wedge 0)$ is also in this ideal. Because $s = (s\vee 0) + (s \wedge 0)$ (see, e.g., \cite[Ch.~XIII, Thm.~7]{Bir79}), we conclude that $s$ lies in the ideal generated by $E$, and so $I$ is generated by $E$.

(4)$\Rightarrow$(5): This is obvious.

(5)$\Rightarrow$(1): Let $I$ be the ideal generated by an end $E$ of $B$. Then $I$ is the set of all $A$-linear combinations of elements of $E$. Let $M$ be a maximal ideal of $B$ containing $E$. It follows from \cite[Prop.~3.11 and Thm.~4.6]{BMMO13a} that $M = P \cap \func{Id}(S)$ for some minimal prime $P$. Then $I \subseteq P$. Therefore, $I \cap A \subseteq P \cap A = 0$. Let $s \in I$. Then $s$ can be written as $s = \sum_{i=1}^n a_i e_i$ with each $e_i \in E$. There are $k_i \in E$ with $e_i \prec k_i$. If $t = \sum_{i=1}^n |a_i| k_i$, then $t \in I$ and $|s| \le \sum_{i=1}^n |a_i|e_i \prec t$. Therefore, $|s| \prec t$. Thus, $I$ is round. Finally, let $J$ be an end of $(S,\prec)$ with $I \subseteq J$. Then $E = I \cap B \subseteq J \cap B$. Because $J$ is an end and we have already proved (3)$\Rightarrow$(4), $J \cap B$ is an end of $(B,\prec)$ and generates $J$ as an ideal. Maximality shows $E = J\cap B$, so $J$ is the ideal generated by $E$. Thus, $J = I$, and so $I$ is an end of $(S,\prec)$.

(5)$\Rightarrow$(2): Let $E$ be an end of $(B,\prec)$ and suppose $I$ is generated as an ideal by $E$. Let $F = \{ b \in B : \lnot b \in E\}$. As we pointed out in Section 5, $F$ is a maximal round filter of $(B,\prec)$, so $\sigma : B \to {\bf 2}$ that sends the members of $F$ to $1$ and the rest of $B$ to $0$ is a de Vries morphism (see, e.g., \cite[Rem.~4.21]{BH13}). By Corollary~\ref{unique extension of map}, $\sigma$ extends uniquely to a proximity morphism $\alpha : S \to A$. We claim that $I = \ker(\alpha)$. To see this, let $s \in I$. Since $I$ is an $\ell$-deal, $|s| \in I$, so $|s|$ is a linear combination of idempotents from $E$. An argument similar to the one in the proof of (1)$\Rightarrow$(3) showing that $0 \le t \in P$ has all the idempotents in its decreasing form in $M$ yields that $|s|$ can be written in decreasing form $|s| = \sum_{i=1}^n b_i e_i$ with each $b_i > 0$ and the $e_i$ in $E$. Since $E$ is a round ideal of $(B,\prec)$, for each $i$ there is $k_i \in E$ with $e_i \prec k_i$. Set $t = \sum_{i=1}^n b_i k_i$. Then $t \in I$ and $|s| \prec t$. Moreover, $\alpha(t) = \sum_{i=1}^n b_i \sigma(k_i) = 0$. Therefore, $s \in \ker(\alpha)$. Conversely, let $s \in \ker(\alpha)$. Then $|s| \prec t$ for some $t$ with $\alpha(t) = 0$. By Lemma~\ref{compatible}, we may write $|s| = \sum_{i=1}^n b_i e_i$ and $t = \sum_{i=1}^n b_i k_i$ in compatible decreasing form with $e_i \prec k_i$ and all $b_i > 0$. Then $0 = \alpha(t) = \sum_{i=1}^n b_i \sigma(k_i)$. Because $\sigma(k_i) \in \{0,1\}$, the condition $\alpha(t) = 0$ forces $\sigma(k_i) = 0$ for each $i$. Thus, $e_i \in \ker(\sigma) = E$. So, $|s| \in I$, and hence $s \in I$ since $I$ is an $\ell$-ideal. Thus, $I = \ker(\alpha)$.
\end{proof}

\begin{remark}
A natural condition to add to the five equivalent conditions of Theorem~\ref{char of ends} would be that $I={\thd}M$ for some maximal $\ell$-ideal $M$ of $S$. While this condition is not equivalent to the others in general, we show that it is equivalent provided $A$ is \emph{Archimedean}; meaning that for each $a,b\in A$, if $na\le b$ for all $n\in\mathbb N$, then $a\le 0$. Indeed, we show that if $A$ is Archimedean and $S$ is a Specker $A$-algebra, then minimal primes in $S$ coincide with maximal $\ell$-ideals of $S$. First suppose that $P$ is a minimal prime ideal of $S$. By \cite[p.~196]{Sub67}, $P$ is an $\ell$-ideal. By the proof of (1)$\Rightarrow$(3) of Theorem~\ref{char of ends}, $S/P \cong A$. Since $A$ is Archimedean, it is simple as an $\ell$-algebra. Thus, $P$ is maximal as an $\ell$-ideal. Conversely, let $M$ be a maximal $\ell$-ideal. Then $M \cap A = 0$ since this intersection is an $\ell$-ideal of $A$ and $A$ is Archimedean, hence simple as an $\ell$-algebra. A Zorn's lemma argument then yields a prime ideal $P \supseteq M$ with $P \cap A = 0$. Since $P \cap A = 0$, by \cite[Lem.~4.5]{BMMO13a}, $P$ is a minimal prime ideal of $S$. Because it is a proper $\ell$-ideal, maximality of $M$ shows $M = P$. Thus, $M$ is a minimal prime ideal of $S$.
\end{remark}

\begin{definition}
Let $(S,\prec)$ be a proximity Specker $A$-algebra.
\begin{enumerate}
\item Let ${\sf End}(S,\prec)$ be the space of ends of $(S,\prec)$ topologized by the basis $\{U(s): s\in S\}$, where $U(s)=\{ I  : s \in I\}$.
\item Let ${\sf Hom}(S,A)$ be the space of proximity morphisms from $S$ to $A$ topologized by the basis $\{V(s) : s\in S\}$, where $V(s)=\{ \alpha : \exists t \text{ with } |s| \prec t \text{ and } \alpha(t) = 0\}$.
\end{enumerate}
\end{definition}

\begin{remark} \label{rem:zeta}
\begin{enumerate}
\item[]
\item  That the sets $U(s)$, $s \in S$, form a basis for ${\sf End}(S,\prec)$ is a consequence of the following easily verifiable facts: $U(s) = U(|s|)$; $U(|s|) \cap U(|t|) = U(|s| \vee |t|)$; $U(1) = \varnothing$; and $U(0) = {\sf End}(S,\prec)$.
\item The same formulas as in (1) hold for $V(s)$, $s\in S$. That $V(|s|) \cap V(|t|) = V(|s| \vee |t|)$ is a consequence of the following lemma, which is parallel to \cite[Lem.~2.2]{Bez12}.
\end{enumerate}
\end{remark}

\begin{lemma}\label{lem:7.7}
Let $\alpha : S \to T$ be a proximity morphism between proximity Specker $A$-algebras.
\begin{enumerate}
\item If $s \in S$, then $\alpha(s) \le -\alpha(-s)$. Therefore, if $s \prec t$, then $\alpha(s) \prec \alpha(t)$.
\item If $s,t,u,v \in S$ with $s \prec u$ and $t \prec v$, then $\alpha(s \vee t) \prec \alpha(u) \vee \alpha(v)$.
\end{enumerate}
\end{lemma}

\begin{proof}
(1) By Proposition~\ref{prop:5.4}, $\alpha|_{\func{Id}(S)}$ is a de Vries morphism. Thus, if $e \in \Id(S)$, then $\alpha(e) \le \lnot \alpha(\lnot e)$. Let $s \in S$. We first assume $s \ge 0$. Write $s = a_0 + \sum_{i=1}^n b_i e_i$ in decreasing form with $b_i > 0$. By Theorem~\ref{characterization of proximity morphism}, $\alpha(s) = a_0 + \sum_{i=1}^n b_i \alpha(e_i)$. Let $b = \sum_{i=1}^n b_i$. As in the proof of Theorem~\ref{induced map on FN}(3), we have
\begin{eqnarray*}
-s &=& -a_0 + \sum_{i=1}^n b_i (-e_i) = -a_0 - b + b + \sum_{i=1}^n b_i (-e_i) \\
&=& (-a_0 -b) + \sum_{i=1}^n b_i(1-e_i) = (-a_0 -b) + \sum_{i=1}^n b_i \lnot e_i.
\end{eqnarray*}
Since $\lnot e_n > \cdots > \lnot e_1$, this writes $-s$ in decreasing form, so $\alpha(-s) = (-a_0 -b) + \sum_{i=1}^n b_i \alpha(\lnot e_i)$. Consequently, $\alpha(-s) = -a_0 + \sum_{i=1}^n b_i(\alpha(\lnot e_i) - 1)$, so $-\alpha(-s) = a_0 + \sum_{i=1}^n b_i (1-\alpha(\lnot e_i)) = a_0 + \sum_{i=1}^n b_i (\lnot\alpha(\lnot e_i))$. Finally, since $\alpha(e_i) \le \lnot\alpha(\lnot e_i)$ for each $i$, we get $\alpha(s) \le -\alpha(-s)$ since all $b_i$ are positive.

For an arbitrary $s\in S$, since $\eta_S(s)$ is a finitely valued function, there is $a \in A$ with $\eta_S(s)+a \ge 0$. This yields $s + a \ge 0$ because $\eta_S$ is an $\ell$-algebra embedding. By the nonnegative case, we have $\alpha(a+s) \le -\alpha(-(a+s))$. This simplifies to $a + \alpha(s) \le a + (-\alpha(-s))$. Consequently, $\alpha(s) \le -\alpha(-s)$. From this we conclude that if $s \prec t$, then $\alpha(s) \le -\alpha(-s) \prec \alpha(t)$, so $\alpha(s) \prec \alpha(t)$.

(2) We have $-\alpha(-s) \prec \alpha(u)$ and $-\alpha(-t) \prec \alpha(v)$, so
\[
-\alpha(-s) \vee -\alpha(-t) \prec \alpha(u) \vee \alpha(v).
\]
But,
\[
-\alpha(-s) \vee -\alpha(-t) = -[\alpha(-s) \wedge \alpha(-t)] = -\alpha(-s \wedge -t) = -\alpha(-(s\vee t)).
\]
Thus, by (1),
\[
\alpha(s \vee t) \le -\alpha(-(s\vee t)) \prec \alpha(u) \vee \alpha(v).
\]
\end{proof}

We are ready to show that for a proximity Specker $A$-algebra $(S,\prec)$, the 1-1 correspondences of Theorem~\ref{char of ends} between the ends of $(S,\prec)$, the ends of $(\Id(S),\prec)$, and the proximity morphisms $S\to A$ extend to the homeomorphisms of the corresponding spaces.

\begin{theorem}\label{homeo of end spaces}
Let $(S,\prec)$ be a proximity Specker $A$-algebra and let $B=\Id(S)$. The spaces ${\sf End}(S,\prec)$ and ${\sf Hom}(S,A)$ are homeomorphic to the compact Hausdorff space $X$ of ends of $(B,\prec)$.
\end{theorem}

\begin{proof}
The theorem is clear if $S$ is trivial. Suppose that $S$ is nontrivial, and define $\varphi:{\sf End}(S,\prec) \rightarrow X$ by $\varphi(I) = I\cap B$ for each $I \in {\sf End}(S,\prec)$. By Theorem~\ref{char of ends}, $\varphi$ is a well-defined bijection. Since $\varphi^{-1}(\zeta(e)) = U(e)$ for each $e\in B$, we have that $\varphi$ is continuous. To see that $\varphi^{-1}$ is continuous, let $s \in S$ and write $s = \sum_{i=1}^n a_i e_i$ in orthogonal form with each $a_i\ne 0$. We show that $U(s) = U(e_1) \cap \cdots \cap U(e_n)$. One inclusion is clear. For the other inclusion, let $I$ be an end of $(S,\prec)$ and let $s \in I$. Then $se_i = a_ie_i$, so $a_i e_i \in I$. Since $I$ is an end, there is $t \in I$ with $|a_ie_i| \prec t$. By Theorem~\ref{char of ends}, $I\cap B$ is an end of $(B,\prec)$ and $I$ is generated by $I\cap B$. Therefore, we may write $t = \sum_{j=1}^m b_j k_j$ with $k_j \in I\cap B$. If $b = \sum_{j=1}^m |b_j|$ and $k = k_1 \vee \cdots \vee k_m$, then $k \in I\cap B$ and $t \le bk$. Thus, $|a_ie_i| \prec bk$. This implies $|a_i|e_i \le bk$, so by Lemma~\ref{lem:3.9}(6), $|a_i|\le b$ and $e_i\le k$, yielding $e_i \in I\cap B$. Since this is true for each $i$, we conclude that $I\in U(e_1) \cap \cdots \cap U(e_n)$, so $U(s) = U(e_1) \cap \cdots \cap U(e_n)$. Therefore, $\varphi(U(s)) = \varphi(U(e_1)) \cap \cdots \cap \varphi(U(e_n))$. Now, if $e \in B$, then $\varphi(U(e)) = \{ I \cap B : e \in I\} = \zeta(e)$. Thus, $\varphi(U(s))$ is open in $X$. Consequently, $\varphi$ is a homeomorphism.

Next, define $\tau:{\sf Hom}(S,A) \to {\sf End}(S,\prec)$ by $\tau(\alpha)=\ker(\alpha)$. By Theorem~\ref{char of ends}, $\tau$ is a well-defined bijection. It is also easy to see that $\tau^{-1}(U(s))=V(s)$ and $\tau(V(s))=U(s)$. Thus, ${\sf Hom}(S,A)$ and ${\sf End}(S,\prec)$ are homeomorphic.
\end{proof}

Consequently, given a proximity Specker $A$-algebra $(S,\prec)$, we can think of the dual compact Hausdorff space of $(S,\prec)$ as either the space of ends of $(S,\prec)$, the space of ends of $(\Id(S),\prec)$, or the space of proximity morphisms $S\to A$.

\section{Categorical considerations}

In this final section we show that the proximity Baer Specker $A$-algebras and proximity morphisms between them form a category, which we denote by ${\bf PBSp}_A$. Using the results obtained in previous sections, we prove that ${\bf PBSp}_A$ is dually equivalent to {\bf KHaus}, thus providing an analogue of de Vries duality for proximity Baer Specker $A$-algebras. As a consequence, we obtain that ${\bf PBSp}_A$ is equivalent to {\bf DeV}.

\begin{proposition} \label{cor:7.8}
The proximity Baer Specker $A$-algebras and proximity morphisms form a category ${\bf PBSp}_A$ where the composition $\beta\star\alpha$ of two proximity morphisms $\alpha:S_1\to S_2$ and $\beta:S_2\to S_3$ is the unique proximity morphism extending the de Vries morphism $\beta|_{\func{Id}(S_2)} \star \alpha|_{\func{Id}(S_1)}$.
\end{proposition}

\begin{proof}
It is easily seen that the identity map on a proximity Baer Specker $A$-algebra is a proximity morphism. By Proposition~\ref{prop:5.4}, if $\alpha$ and $\beta$ are proximity morphisms, then their restrictions to the idempotents are de Vries morphisms. Therefore, $\beta|_{\func{Id}(S_2)}\star\alpha|_{\func{Id}(S_1)}$ is a de Vries morphism. Thus, by Corollary~\ref{unique extension of map}, $\beta \star \alpha$ is a proximity morphism. We show that $\star$ is associative. Suppose that $\alpha_1:S_1\to S_2$, $\alpha_2:S_2\to S_3$, and $\alpha_3:S_3\to S_4$ are proximity morphisms. Since the restrictions of $\alpha_3\star(\alpha_2\star\alpha_1)$ and $(\alpha_3\star\alpha_2)\star\alpha_1$ to the idempotents are de Vries morphisms, we have that they are equal. Applying Corollary~\ref{unique extension of map} again yields $\alpha_3 \star (\alpha_2 \star \alpha_1) = (\alpha_3 \star_2) \star \alpha_1$. Thus, the proximity Baer Specker $A$-algebras with proximity morphisms form a category.
\end{proof}

\begin{remark}
It would seem more natural to first introduce the category ${\bf PSp}_A$ of proximity Specker $A$-algebras, and treat ${\bf PBSp}_A$ as a full subcategory of ${\bf PSp}_A$. Similarly, it would seem more natural to first introduce the category {\bf PBA} of proximity Boolean algebras, and treat {\bf DeV} as a full subcategory of {\bf PBA}. However, if $\alpha : B_1 \to B_2$ and $\beta : B_2 \to B_3$ are proximity morphisms between proximity Boolean algebras, then the formula $(\beta \star \alpha)(a) = \bigvee \{ \beta(\alpha(b)) : b \prec a\}$ need not be well defined because, if $C$ is not complete, then the join may not exist. It is for this reason that we do not talk categorically about proximity Boolean algebras and proximity Specker $A$-algebras.
\end{remark}

Next we show that although proximity morphisms are not $A$-algebra homomorphisms, proximity isomorphisms are $A$-algebra isomorphisms that preserve and reflect proximity. This is parallel to what happens in {\bf DeV} \cite[Ch.~I.5]{deV62}.

\begin{lemma}\label{proximity isomorphism}
Let $(S,\prec),(T,\prec)\in{\bf PBSp}_A$ and let $\alpha : S \to T$ be a proximity morphism. Then $\alpha$ is an isomorphism in ${\bf PBSp}_A$ iff $\alpha$ is an $A$-algebra isomorphism such that $s \prec t$ in $S$ iff $\alpha(s) \prec \alpha(t)$ in $T$.
\end{lemma}

\begin{proof}
First suppose that $\alpha : S \to T$ is an $A$-algebra isomorphism such that $s \prec t$ iff ${\alpha(s) \prec \alpha(t)}$. Since $\alpha$ is an $A$-algebra homomorphism, $-\alpha(-s) = \alpha(s)$. As $s \prec t$ implies $\alpha(s) \prec \alpha(t)$, we obtain that $s\prec t$ implies $-\alpha(-s)\prec \alpha(t)$. Consequently, to see that $\alpha$ is a proximity morphism, we only need to check that $\alpha(t)$ is the least upper bound of $\{\alpha(s) : s \prec t\}$. By \cite[Cor.~5.3]{BMMO13a}, $\alpha$ is an $\ell$-algebra isomorphism, so $\alpha$ is order preserving. Therefore, $\alpha(t)$ is an upper bound of $\{\alpha(s) : s \prec t \}$. Let $r$ be an upper bound of this set. Then $\alpha(s) \le r$ for all $s$ with $s \prec t$. As $\alpha$ is an $A$-algebra isomorphism, so is $\alpha^{-1}$. Therefore, $\alpha^{-1}$ is order preserving, and $\alpha(s) \le r$ implies $s \le \alpha^{-1}(r)$. By Lemma~\ref{lem:6.1}, $t$ is the least upper bound of all elements proximal to it. Thus, $t \le \alpha^{-1}(r)$, and so $\alpha(t) \le r$. This proves that $\alpha(t)$ is the least upper bound of $\{\alpha(s) : s \prec t\}$. It follows that $\alpha$ is a proximity morphism. The same argument shows that $\alpha^{-1}$ is a proximity morphism. To see that $\alpha^{-1} \star \alpha = \func{id}_S$, let $s\in S$ and write $s = a_0 + \sum_{i=1}^n b_ie_i$ in decreasing form. Then, using Theorem~\ref{characterization of proximity morphism} and the fact that the restriction of $\alpha^{-1}\star\alpha$ to the idempotents is the identity on the idempotents, we obtain $(\alpha^{-1}\star\alpha)(s) = a_0 + \sum_{i=1}^n b_i(\alpha^{-1}\star\alpha)(e_i) = a_0 + \sum_{i=1}^n b_ie_i = s$. Therefore, $\alpha^{-1} \star \alpha = \func{id}_S$. A similar argument shows $\alpha \star \alpha^{-1} = \func{id}_T$. Thus, $\alpha$ is a proximity isomorphism.

Next suppose that $\alpha: S \to T$ is a proximity isomorphism. Then there is a proximity isomorphism $\beta : T \to S$ such that $\beta \star \alpha$ is the identity on $S$ and $\alpha \star \beta$ is the identity on $T$. By the definition of $\star$, we see that the restrictions of $\alpha$ and $\beta$ to the idempotents are inverse de Vries isomorphisms, so $\alpha|_{\func{Id}(S)}$ and $\beta|_{\func{Id}(T)}$ are inverse Boolean isomorphisms. By \cite[Lem.~2.5 and Thm.~2.7]{BMMO13a}, there is a unique $A$-algebra homomorphism $\alpha^\prime : S \to T$ extending $\alpha|_{\func{Id}(S)}$. We show that $\alpha^\prime = \alpha$. Let $s\in S$ and write $s = a_0 + \sum_{i=1}^n b_i e_i$ in decreasing form. By Theorem~\ref{characterization of proximity morphism}, $\alpha(s) = a_0 + \sum_{i=1}^n b_i \alpha(e_i)$. Since $\alpha^\prime$ is an $A$-algebra homomorphism, we also have $\alpha^\prime(s) = a_0 + \sum_{i=1}^n b_i \alpha(e_i)$. Thus, $\alpha^\prime=\alpha$, and so $\alpha$ is an $A$-algebra homomorphism. By the same reasoning, $\beta$ is an $A$-algebra homomorphism. It follows that $\beta\circ\alpha$ is the identity on $S$ because it is an $A$-algebra endomorphism which is the identity on $\func{Id}(S)$, a generating set of $S$ as an $A$-algebra. Similarly, $\alpha\circ\beta$ is the identity on $T$. Consequently, $\alpha$ is an $A$-algebra isomorphism whose inverse is $\beta$, and $s \prec t$ iff $\alpha(s) \prec \alpha(t)$ because both $\alpha$ and $\beta$ preserve proximity.
\end{proof}

Next we construct contravariant functors $(-)_*:{\bf PBSp}_A \rightarrow {\bf KHaus}$ and $(-)^*:{\bf KHaus} \rightarrow {\bf PBSp}_A$ that yield a dual equivalence of ${\bf PBSp}_A$ and {\bf KHaus}.

Define a contravariant functor $(-)^*:{\bf KHaus}\to{\bf PBSp}_A$ as follows. For $X\in{\bf KHaus}$, let $X^*=(FN(X,A),\prec_X)$ be the de Vries power of $A$ by $X$; and for a continuous map $\varphi:X\to Y$, let $\varphi^*:FN(Y)\to FN(X)$ be the proximity morphism given by $\varphi^*(f)=(f\circ\varphi)^\#$. By Theorem~\ref{FN algebra} and Proposition~\ref{characterization of proximity morphism}, $(-)^*:{\bf KHaus}\to{\bf PBSp}_A$ is a well-defined contravariant functor.

Define a contravariant functor $(-)_*:{\bf PBSp}_A\to{\bf KHaus}$ as follows. For $(S,\prec) \in {\bf PBSp}_A$, let $S_*$ be the space of ends of $(S,\prec)$. By Theorem~\ref{homeo of end spaces}, $S_*\in{\bf KHaus}$. For a proximity morphism $\alpha : S \to T$, let $\alpha_* : T_* \to S_*$ be given by $\alpha_*(I) = {\thd} \alpha^{-1}(I)$. That $\alpha_*$ is a well-defined continuous map is proved in the next lemma.

\begin{lemma} \label{lem:8.15}
Let $\alpha : S \to T$ be a proximity morphism between proximity Specker $A$-algebras $(S,\prec)$ and $(T,\prec)$. Define $\alpha_* : T_* \to S_*$ by $\alpha_*(I) = {\thd} \alpha^{-1}(I)$. Then $\alpha_*$ is a well-defined continuous map.
\end{lemma}

\begin{proof}
If $T$ is trivial, then there is nothing to prove. Suppose that $T$ is nontrivial. Let $I$ be an end of $T$. By Lemma~\ref{kernel is round}, ${\thd} \alpha^{-1}(I)$ is a proper round ideal of $S$. To see that ${\thd} \alpha^{-1}(I)$ is an end, let $J$ be an end of $S$ containing ${\thd} \alpha^{-1}(I)$. By Theorem~\ref{char of ends}, $J\cap\Id(S)$ is an end of $\Id(S)$ and $J$ is generated by $J\cap\Id(S)$. The same reasoning gives $I\cap\Id(T)$ is an end of $\Id(T)$ and $I$ is generated by $I\cap\Id(T)$. We have
$$
\left({\thd} \alpha^{-1}(I\cap\Id(T)) \right)\cap\Id(S)\subseteq {\thd} \alpha^{-1}(I)\cap\Id(S) \subseteq J\cap\Id(S).
$$
By Proposition~\ref{prop:5.4}, the restriction of $\alpha$ to $\Id(S)$ is a de Vries morphism. Therefore, ${\thd} \alpha^{-1}(I\cap\Id(T))\cap\Id(S)$ is an end of $\Id(S)$, and so ${\thd} \alpha^{-1}(I\cap\Id(T))\cap\Id(S)=J\cap\Id(S)$. This implies that $J$ is generated by ${\thd} \alpha^{-1}(I)\cap\Id(S)$. Thus, $J\subseteq\alpha^{-1}(I)$, yielding that $\alpha_*(I) = {\thd} \alpha^{-1}(I)$ is an end.

To show that $\alpha_*$ is continuous, it is sufficient to see that $\alpha_*^{-1}(U(s)) = \bigcup \{ U(\alpha(t)) : |s| \prec t\}$. Indeed, $I \in \alpha_*^{-1}(U(s))$ iff $s \in {\thd} \alpha^{-1}(I)$, which happens iff there is $t$ with $|s| \prec t$ and $\alpha(t) \in I$, which is true iff $I \in \bigcup\{U(\alpha(t)) : |s| \prec t\}$.
\end{proof}

This implies that $(-)_*:{\bf PBSp}_A\to{\bf KHaus}$ is a well-defined contravariant functor.

\begin{theorem}\label{PBSp and KHaus}
The functors ${(-)}_*$ and ${(-)}^*$ yield a dual equivalence of ${\bf PBSp}_A$ and {\bf KHaus}.
\end{theorem}

\begin{proof}
By Corollary~\ref{1-1 correspondence}, for $(S,\prec)\in{\bf PBSp}_A$ and $X\in{\bf KHaus}$, we have $\hom_{{\bf PBSp}_A}(S,X^*)\simeq\hom_{{\bf KHaus}}(X,S_*)$. It follows from the proof of Corollary~\ref{1-1 correspondence} that the bijection is natural. Therefore, ${(-)}_*$ and ${(-)}^*$ define a contravariant adjunction between ${\bf PBSp}_A$ and {\bf KHaus}. Let $(S,\prec)\in{\bf PBSp}_A$ and let $X$ be the space of ends of $(\Id(S),\prec)$. By Theorem~\ref{homeo of end spaces}, $S_*$ is homeomorphic to $X$, so $(S_*)^*$ is isomorphic to $(FN(X),\prec)$. By Corollary~\ref{FN algebra char} and Lemma~\ref{proximity isomorphism}, the $\ell$-algebra embedding $\eta:S\to FN(X)$ of Theorem~\ref{rep} is an isomorphism in ${\bf PBSp}_A$. Thus, the unit of the contravariant adjunction is an isomorphism.

Let $X\in{\bf KHaus}$. By Theorem~\ref{FN algebra}, $X^*$ is a proximity Baer Specker $A$-algebra. By Lemma~\ref{idempotents of FN}, $\mathcal{RO}(X)$ is isomorphic to $\Id(FN(X))$. Therefore, by Theorem~\ref{homeo of end spaces}, $(X^*)_*$ is homeomorphic to the space of ends of $\mathcal{RO}(X)$. By de Vries duality, $X$ is homeomorphic to the space of ends of $\mathcal{RO}(X)$, so $X$ is homeomorphic to $(X^*)_*$. Thus, the counit of the contravariant adjunction is an isomorphism. Consequently, ${\bf PBSp}_A$ is dually equivalent to {\bf KHaus}.
\end{proof}

\begin{corollary}\label{equivalence}
${\bf PBSp}_A$ is equivalent to {\bf DeV}.
\end{corollary}

\begin{proof}
Combine Theorem~\ref{PBSp and KHaus} with de Vries duality.
\end{proof}

\begin{remark}
The equivalence of ${\bf PBSp}_A$ and {\bf DeV} can be established directly through the covariant functors $\mathcal{I}:{\bf PBSp}_A\to{\bf DeV}$ and $\mathcal{S}:{\bf DeV} \to {\bf PBSp}_A$, where $\mathcal{I}$ sends a proximity Baer Specker $A$-algebra $(S,\prec)$ to the de Vries algebra of idempotents of $(S,\prec)$, while $\mathcal{S}$ sends a de Vries algebra $(B,\prec)$ to the de Vries power of $A$ by $(B,\prec)$.
\end{remark}

\begin{remark}\label{rem:8.8}
Following \cite[Def.~4.5]{Bez10}, we call a de Vries algebra $(B,\prec)$ \emph{zero-dimensional} provided $a \prec b$ implies that there is $c \in B$ with $c \prec c$ and $a \prec c \prec b$. Let ${\bf zDeV}$ be the full subcategory of {\bf DeV} whose objects are zero-dimensional de Vries algebras, and let {\bf Stone} be the full subcategory of {\bf KHaus} whose objects are Stone spaces (zero-dimensional compact Hausdorff spaces). By \cite[Thm.~4.12]{Bez10}, ${\bf zDeV}$ is dually equivalent to ${\bf Stone}$.

Analogously, we call a proximity Specker $A$-algebra $(S,\prec)$ \emph{zero-dimensional} provided $s \prec t$ implies that there is $r \in S$ with $r \prec r$ and $s \prec r \prec t$. Let ${\bf zPBSp}_A$ be the full subcategory of ${\bf PBSp}_A$ of zero-dimensional proximity Baer Specker $A$-algebras. It is a consequence of Theorem~\ref{PBSp and KHaus}, Corollary~\ref{equivalence}, and \cite[Thm.~4.12]{Bez10} that ${\bf zPBSp_A}$ is equivalent to ${\bf zDeV}$ and is dually equivalent to ${\bf Stone}$. Thus, by \cite[Thm.~4.9]{Bez10}, ${\bf zPBSp_A}$ is a coreflective subcategory of ${\bf BPSp_A}$.
\end{remark}

\begin{remark}
Following \cite[Sec.~5]{Bez10}, we call a de Vries algebra $(B,\prec)$ \emph{extremally disconnected} provided $a \prec b$ iff $a\le b$. Let ${\bf eDeV}$ be the full subcategory of {\bf DeV} whose objects are extremally disconnected de Vries algebras, and let {\bf ED} be the full subcategory of {\bf KHaus} whose objects are extremally disconnected compact Hausdorff spaces. Then ${\bf eDeV}$ is a full subcategory of {\bf zDeV}, {\bf ED} is a full subcategory of {\bf Stone}, ${\bf eDeV}$ is isomorphic to the category {\bf cBA} of complete Boolean algebras and Boolean homomorphisms, and {\bf eDeV} is dually equivalent to ${\bf ED}$ \cite[Sec.~6.2]{Bez10}.

Analogously, we call a proximity Specker $A$-algebra $(S,\prec)$ \emph{extremally disconnected} provided $s \prec t$ iff $s \le t$. Let ${\bf ePBSp}_A$ be the full subcategory of ${\bf PBSp}_A$ of extremally disconnected proximity Baer Specker $A$-algebras. Then ${\bf ePBSp}_A$ is a full subcategory of ${\bf zPBSp}_A$ and is isomorphic to the category ${\bf BSp}_A$ of Baer Specker $A$-algebras and $A$-algebra homomorphisms. Thus, by \cite[Thm.~4.7]{BMMO13a}, ${\bf ePBSp}_A$ is dually equivalent to ${\bf ED}$. In addition, ${\bf ePBSp_A}$ is equivalent to ${\bf eDeV}$.
\end{remark}

\begin{remark}
\begin{enumerate}
\item[]
\item For a compact Hausdorff space $X$, we recall \cite{Gle58} that the \emph{Gleason cover} $Y$ of $X$ is the Stone space of $\mathcal{RO}(X)$. Therefore, $\mathcal{RO}(X)$ is isomorphic to the Boolean algebra ${\sf Clopen}(Y)$ of clopen subsets of $Y$. Since up to isomorphism, $FN(Y)$ is generated by $\mathcal{RO}(X)$ and $FC(Y)$ is generated by ${\sf Clopen}(Y)$, we obtain that $FN(X)$ is isomorphic to $FC(Y)$. This yields an alternate representation of finitely valued normal functions on $X$.
\item For a de Vries algebra $(B,\prec)$, we recall \cite[Sec.~7]{Bez10} that the Stone space of $B$ is the Gleason cover of the space of ends of $(B,\prec)$. Similarly, for a proximity Baer Specker $A$-algebra  $(S,\prec)$, the Gleason cover of the space of ends of $(S,\prec)$ can be constructed as the minimal prime spectrum of $S$. For an Archimedean $A$, the Gleason cover can alternately be constructed as the space of maximal $\ell$-ideals of $S$.
\end{enumerate}
\end{remark}

\bibliographystyle{amsplain}
\bibliography{Gelfand}

\bigskip

\noindent Department of Mathematical Sciences, New Mexico State University, Las Cruces NM 88003-8001

\bigskip

\noindent Dipartimento di Matematica ``Federigo Enriques," Universit\`a degli Studi di Milano, via Cesare Saldini 50, I-20133 Milano, Italy

\bigskip

\noindent gbezhani@nmsu.edu, vincenzo.marra@unimi.it, pmorandi@nmsu.edu, oberdin@nmsu.edu

\end{document}